\documentclass[12pt,a4paper]{amsart}
\usepackage{amssymb}
\usepackage{amsfonts}
\usepackage{amsthm}
\usepackage{amsmath}
\usepackage{amscd}
\usepackage[utf8]{inputenc}
\usepackage{t1enc}
\usepackage[mathscr]{eucal}
\usepackage{indentfirst}
\usepackage{graphicx}
\usepackage{graphics}
\numberwithin{equation}{section}
\usepackage[margin=2.9cm]{geometry}
\usepackage{epstopdf} 
\usepackage{xcolor}

\usepackage{verbatim}
\usepackage{enumitem}
\usepackage{hyperref}

\theoremstyle{plain}
\newtheorem{theorem}{Theorem}[section]
\newtheorem{lemma}[theorem]{Lemma}
\newtheorem{corollary}[theorem]{Corollary}
\newtheorem{proposition}[theorem]{Proposition}
\newtheorem{condition}[theorem]{Condition}

\theoremstyle{definition}
\newtheorem{definition}[theorem]{Definition}

\newtheorem{remark}[theorem]{Remark}

\newtheorem{example}[theorem]{Example}
\newcommand{\supp}{\operatorname{supp}}

\newcommand{\R}{\mathbb{R}}

\newcommand{\De}{\Delta}

\newcommand{\Eb}{\mathbf{E}}
\newcommand{\E}{\mathbf{E}}

\newcommand{\Nb}{\mathbf{N}}

\newcommand{\Ht}{\mathrm{H}}
\newcommand{\htt}{\mathrm{h}}

\newcommand{\Bc}{\mathcal{B}}
\newcommand{\Mt}{\mathrm{M}}

\let\emptyset\varnothing

\font\sc=rsfs10 at 12pt
\newcommand{\Ls}{\sc\mbox{L}\hspace{1.0pt}}

\title[Negative eigenvalue estimates for measure potentials]{Negative eigenvalue estimates for polyharmonic Schr\"odinger operators 
with measure-potentials: the subcritical case}

\author[]{Medet Nursultanov}
\address {Department of Mathematics and Statistics, University of Helsinki, Finland, and Institute of Mathematics and Mathematical 
Modeling, Almaty, Kazakhstan}
\email{medet.nursultanov@gmail.com}

\author[]{Grigori Rozenblum}
\address{Chalmers University of Technology, School of Mathematics, Gothenburg, Sweden}
\email{grigori@chalmers.se}

\subjclass[2020]{35P15, 35J30; 47A75, 31C15, 81Q10, 35J10}

\keywords{polyharmonic Schr\"odinger operators, measure potentials, negative eigenvalues, Otelbaev function, Lieb-Thirring 
inequalities}
\begin{document}

\begin{abstract}
We study spectral estimates for polyharmonic Schr\"odinger operators $-\Delta^l-\mu$ in the subcritical regime \(2l<\Nb\).
The measure potential $\mu$ is assumed to satisfy a capacitary smallness condition
which guarantees that the corresponding operator is semibounded and self-adjoint. With such a measure $\mu$ we associate an
Otelbaev function, which reflects both the local concentration and the
spatial distribution of the potential. In terms of this function, we obtain two-sided estimates for the distribution function of 
the negative eigenvalues, and derive a sufficient condition and a necessary condition for the discreteness of the negative 
spectrum. As an application, we establish  two-sided  estimates of Lieb-Thirring-type, improving the classical ones.
\end{abstract}

\maketitle

\tableofcontents

\section{Introduction}
Spectral estimates for Schr{\"o}dinger operators and their generalizations form a traditional topic in Mathematical Physics and 
remain an active area of research; see, in particular, the recent  fundamental monograph \cite{FrankLaptevWeidl}. Best elaborated are 
phase 
volume estimates, such as the 
Cwikel-Lieb-Rozenblum and Lieb-Thirring ones which are sharp in large coupling parameters and in the quasi-classical behavior, being 
supported by asymptotic formulas. An extensive bibliography on this topic can be found in \cite{FrankLaptevWeidl}.

With all aesthetical value of phase volume estimates, there is an essential shortcoming. In fact, in all  results in this line, 
spectral 
characteristics of the operator are estimated  in terms of integral of the potential taken to some power. This means that for all 
equimeasurable potentials, the same estimate is declared;  thus, it follows that these estimates cannot reflect the spatial 
distribution of the potentials. 

In this paper we deal with spectral estimates that aim for taking into account this spatial distribution. We consider 
Schr{\"o}dinger-type 
operators in $L^2(\R^\Nb)$, formally acting as
\begin{equation}\label{Hmu}
  H u\equiv H_\mu u=(-\De)^lu-\mu u,
\end{equation}
with integer $l$, $2l<\Nb$, with a non-negative Radon measure $\mu$. The conditions imposed on $\mu$ guarantee that $H_{\mu}$ can 
be defined as a semi-bounded self-adjoint operator, using the quadratic form initially defined on $C_0^{\infty}(\R^{\Nb})$. 
Our approach is based upon introducing a certain averaging of the measure $\mu$. This idea was actually used rather long ago for 
$\mu$ being absolutely continuous with respect to the Lebesgue measure.  Probably, it was originally  implemented by M.Otelbaev in 
\cite{O76} in the one-dimensional case, see details in 
his book \cite{O90}, where more applications, involving also  higher order operators, were developed. This averaging 
procedure was further extended by M.Otelbaev to the multi-dimensional case, see \cite{O81}, \cite{O83}, and since then, the name 'the 
Otelbaev function' was being used. Independently, a close version of averaging  was developed for general elliptic equations by 
C.Fefferman, see \cite{Fefferman83}; after the extension by Z.Shen, see \cite{Shen96}, to magnetic operators, this averaged function 
is often used under the name Fefferman-Shen (FS) maximal function. Recently, a Fefferman-type approach was further developed in the higher-order setting in \cite{ZhaoTang2024}.

Compared with other traditional averaging methods, these averaged functions have variable size window adapted to the size of the 
'potential' and they  sleeken the potential of the Sturm-Liouville operator (or, the symbol of a more general elliptic operator). 
Thus,  eigenvalue estimates, expressed in these terms, are sharper than the phase volume ones.

Rather recently, for the Schr{\"o}dinger operator $-\De-V(x)$, a relation was found of FS maximal function with the \emph{landscape 
function}, the solution of the equation $(-\De-V(x))u(x) = Mu(x)$, rather widely used in analysis of quantum Hamiltonians, see 
\cite{DFM}, \cite{BFS2023}, \cite{BFS2024}. This approach enabled the authors of these papers to obtain new eigenvalue estimates for 
rather 
singular potentials belonging to the Kato class, however still corresponding to usual potentials, in other words, with measure $\mu$ 
in 
\eqref{Hmu} being absolutely continuous with respect to the Lebesgue measure. Moreover, this approach is  restricted to the classical 
Schr{\"o}dinger operator, since it is essentially based upon the maximum principle and other special features  of second order 
operators.  

Singular potentials have been the object of spectral studies since rather long ago. In particular, Schr{\"o}dinger and Dirac 
operators with singular interactions, namely, measures supported on points or surfaces, were suggested in the physics literature  
probably, since  \cite{Bethe},  see  an extensive literature cited in \cite{BFKLR} for further developments. For more general 
singular measures, supported on  general sets of Lebesgue measure zero the spectral theory was developed on the base of capacity 
considerations, see, especially, \cite{Mazya}.

In the present paper,  in Sect.2, after introducing main notations, we  describe which measures can act as potential for the 
poly-harmonic Schr{\"o}dinger operator and explain capacity considerations lying in the base of our constructions. Then, in Sect.3, 
we associate with a given measure $\mu$ and the polyharmonic operator 
$(-\Delta)^{l}$, the  Otelbaev function $q^*_{\mu,\tau}(x)$ depending on a positive parameter $\tau$. This function generalizes the 
one introduced for an absolutely continuous $\mu$ in \cite{O81}, \cite{O83} and the one studied in the one-dimensional case for 
arbitrary measures $\mu$ in \cite{FulscheNursultanovRozenblum2025}.  For each  $x\in \mathbb{R}^\Nb$, the value  $q^*_{\mu,\tau}(x)$ 
reflects both the size of $\mu$ in a  neighborhood of $x$ and the spreading of $\mu$ far away from $x$. We also establish here some 
important 
properties of the Otelbaev function, which are further  used for obtaining eigenvalue estimates.

Main results of the paper are presented in Sect.4. There, we establish upper and lower  estimates for $N(-\lambda;H_\mu)$, the number 
of points of the 
negative spectrum of the $l$-harmonic Schr\"odinger operator, lying below the reference level $-\lambda<0$, in the terms of the 
Otelbaev function. One can observe the relation of these estimates with the classical estimates for this quantity, following from the 
CLR bound. Namely, in the latter bound, with an absolutely continuous measure $\mu$, the quantity  $N(-\lambda;H_\mu)$ is 
majorated by the volume of the domain in the phase space $\mathbb{R}^{\Nb}\times \mathbb{R}^{\Nb}$ where the classical Hamiltonian is 
less than $-\lambda$. 
In our case, with, possibly, singular measure $\mu$, the Otelbaev function, inserted on the place of potential,  produces an 
effective Hamiltonian, and our result, Theorem \ref{t_est_distr_func_LO}, declares that $N(-\lambda;H_\mu)$ is bounded by a similar 
phase volume for the effective Hamiltonian (however, with a scaled reference level). As qualitative consequences of Theorem 
\ref{t_est_distr_func_LO} we find (separately) sufficient and necessary  
conditions for discreteness of the negative spectrum. 

We must note  that  we do not touch upon here the problem of the finiteness of the negative spectrum and the estimates of the number 
of negative eigenvalues $N(0;H_\mu)$. It is known that even in the classical setting, as soon as we go beyond the CLR type bound, 
these estimates are considerably more delicate 
than the ones for $-\lambda<0$, see \cite{BirmanSolomyak1991}, \cite{BirmanSolomyak1992}. Even in the one-dimensional case, as shown 
in \cite{FulscheNursultanovRozenblum2025}, the estimate via the phase space volume for the effective Hamiltonian for $\lambda=0$ is 
trivially void since this latter volume is infinite due to the insufficient decay rate of the Otelbaev function at infinity. An 
adaptation of the Otelbaev function approach to this question will be published on another occasion.

The study of the Lieb-Thirring estimates for Schr\"odinger-type operators,  i.e., estimates for the sum of powers of negative 
eigenvalues, is an important topic in spectral theory, see \cite{FrankLaptevWeidl} 
and numerous literature cited there.
Our Sect.5 is devoted to establishing such estimates for Schr\"odinger-type operators with measure-potential. We prove  upper
LT-type estimates,  again in terms of the \emph{effective} Hamiltonian involving the Otelbaev function; our approach produces also 
lower estimates in the same terms.

It is important to compare our results with the known estimates, as long as the latter ones are available. Even in the classical 
setting, for a function-potential,  the classical LT inequality follows from ours, (if one does not 
care for sharp values of the constant in the estimates). Moreover, this implication is sharp: we present examples where the effective 
phase space quantity is finite while the right-hand side in the classical LT estimate is infinite. For singular measures, LT-type 
estimates were established recently in
\cite{Rozenblum2022} (see also \cite{FrankLaptev}) for measures of the form
$\mu=V\nu$, where $\nu$ satisfies an upper  Ahlfors-type condition and $V$ belongs to
a suitable Lebesgue space with respect to $\nu$. In Sect.6,  we compare these estimates
with our Otelbaev-function bounds. Most interesting is the case when the measure $\mu$ is supported on a Lipschitz surface of 
codimension less than $2l$.  For Ahlfors regular measures
we show that the known surface Lieb--Thirring bounds majorize the corresponding
Otelbaev-function bounds. Thus, in this setting, the previous estimates follow
from our general result. We also give examples showing that the converse
implications  may fail in general.

Our considerations concern the so-called subcritical case, $2l<\Nb$. It is known that even in the classical setting,  the spectral 
analysis in the critical ($2l=\Nb$) and the supercritical ($2l>\Nb$) cases differ essentially from the subcritical setting. Here, 
although the general idea of the Otelbaev function persists, quite a lot of new circumstances arise. We plan to discuss these cases 
on a later occasion.

\section{Notations and setting} 
\subsection{Main notations}
Throughout the paper, $\Nb,l\in\mathbb N$ are fixed and \emph{always} satisfy $2l<\Nb$.
For nonnegative quantities $A$ and $B$, we write $A\lesssim B$ if
$A\le C B$ with a constant $C>0$ independent of the  parameters involved.
We write $A\gtrsim B$ if $B\lesssim A$, and $A\simeq B$ if both
$A\lesssim B$ and $A\gtrsim B$ hold.

By $\mathbb Q$ we denote the collection of all open cubes in $\mathbb R^\Nb$
whose edges are parallel to the coordinate axes chosen in a special way and then fixed; see  Proposition \ref{coor_syst}.
For $x\in\mathbb R^\Nb$ and $d>0$, $Q_d(x)\in\mathbb Q$ denotes the cube
centred at $x$ with edgelength $d$.
When the centre is irrelevant, we write simply $Q_d$.

For $Q\in\mathbb Q$ with edgelength $d>0$ and center $x$, and for $c>0$, 
we denote by $cQ$ the cube in $\mathbb Q$ with center $x$ and side length $cd$, that is,
\begin{equation*}
    cQ:=Q_{cd}(x).
\end{equation*}

For a multi-index $\alpha=(\alpha_1,\dots,\alpha_\Nb)$, we use the standard
notation $|\alpha|=\alpha_1+\cdots+\alpha_\Nb$ and $D^\alpha
=\partial_{x_1}^{\alpha_1}\cdots\partial_{x_\Nb}^{\alpha_\Nb}$.
For a function $u$ on $\mathbb R^\Nb$, we denote by
\[
\nabla^l u=\{D^\alpha u:\ |\alpha|=l\}
\]
the collection of all weak partial derivatives of order $l$, and set

\[
\|\nabla^l u\|_{L^2(\Omega)}
=
\left(
\sum_{|\alpha|=l}\binom{l}{\alpha}
\int_\Omega |D^\alpha u(x)|^2\,dx
\right)^{1/2},
\qquad
\binom{l}{\alpha}:=\frac{l!}{\alpha_1!\cdots\alpha_{\Nb}!}.
\]

For an open set $\Omega\subset\mathbb R^\Nb$, the \emph{Sobolev space} is defined as
\[
\Ht^l(\Omega)
=
\bigl\{
u:\ \|\nabla^l u\|_{L^2(\Omega)}+\|u\|_{L^2(\Omega)}<\infty
\bigr\}.
\]
The space $\htt^l(\mathbb R^\Nb)$ (the \emph{homogeneous Sobolev space}) is defined as the completion of
$C_0^\infty(\mathbb R^\Nb)$ with respect to the seminorm
$\|\nabla^l(\cdot)\|_{L^2(\mathbb R^\Nb)}$. Due to the Hardy inequality, $\htt^l(\mathbb R^\Nb)$ is a space of functions and the 
above seminorm is, in fact, a norm in this space.

\subsection{Capacities and measures} In our analysis, an essential role will be played by capacity considerations; we refer the 
reader to the book \cite{Mazya} for a 
detailed exposition. The main facts we use are the following.

For a compact set $\Eb\subset\mathbb R^\Nb$, the \emph{homogeneous Sobolev capacity}
$\operatorname{Cap}_l(\Eb)$ is defined by
\[
\operatorname{Cap}_l(\Eb)
=
\inf\bigl\{
\|\nabla^l u\|_{L^2(\mathbb R^\Nb)}^2:
u\in C_0^\infty(\mathbb R^\Nb),\ u\ge1 \text{ on } \Eb
\bigr\}.
\]
For an arbitrary set $\Eb\subset\mathbb R^\Nb$, the inner and outer capacities
are defined by
\begin{align*}
    &\underline{\operatorname{Cap}}_l(\Eb) = 
    \sup\{\operatorname{Cap}_l(K): K\subset \E,\ K \text{ compact}\},\\
    &\overline{\operatorname{Cap}}_l(\Eb) =
    \inf\{\operatorname{Cap}_l(G): \Eb\subset G,\ G \text{ open}\}.
\end{align*}
For  Borel sets, the inner and outer capacities
coincide, and we write simply $\operatorname{Cap}_l(\Eb)$.

We also use the Riesz potential of order $l$, defined for locally integrable
functions $f$ on $\mathbb R^\Nb$ by
\[
I_l f(x)
=
\int_{\mathbb R^\Nb} \frac{f(y)}{|x-y|^{\Nb-l}}\,\mathrm{d}y,
\qquad x\in\mathbb R^\Nb.
\]
The corresponding \emph{Riesz capacity} $\operatorname{Cap}_l^R(\Eb)$ is defined by
\[
\operatorname{Cap}_l^R(\Eb)
=
\inf\bigl\{
\|f\|_{L^2(\mathbb R^\Nb)}^2:
f\ge0,\ I_l f\ge1 \text{ on } \E
\bigr\}.
\]

\begin{remark}\label{rem:S_R_capacities}
    In our conditions, the homogeneous {Sobolev capacity } $\operatorname{Cap}_l$
    and the Riesz capacity $\operatorname{Cap}_l^R$ of Borel sets differ only by a
    multiplicative constant, that is,
    \begin{equation*}
        \operatorname{Cap}_l(\Eb)=C_{\Nb,l}\,\operatorname{Cap}_l^R(\Eb),
    \end{equation*}
    for some constant $C_{\Nb,l}>0$; see \cite[Section 10.4.1]{Mazya}.
\end{remark}

All measures considered in the paper are Radon measures on $\mathbb R^\Nb$.
For a $\mu$-measurable set $\Omega\subset\mathbb R^\Nb$, $\mu\arrowvert_\Omega$ denotes the
restriction of measure $\mu$ to $\Omega$. In some cases, the measure $\mu$ has the form $\mu=V\nu$, where $\nu$ is some fixed Radon 
measure and the density  $V$ is  variable, belonging to some Lebesgue space $L^p(\Omega,\nu)$ with respect to $\nu$, $1\le p\le\infty$ .

The following proposition collects important properties of the capacitary potential of a compact set that will be needed later.

\begin{proposition}\label{prop:cappot}
There exists a constant $C_0=C_0(\Nb,l)>0$ with the following property. For every
compact set $\Eb\subset\R^\Nb$ with $\operatorname{Cap}_l(\Eb)>0$ there exists
$u_\Eb\in\htt^l(\R^\Nb)$ such that
\begin{enumerate}
    \item\label{cp:one} $u_\Eb\ge 1$ $\operatorname{Cap}_l$-quasi-everywhere on $\Eb$;
      that is, $u_\Eb\ge 1$ on $\Eb\setminus e$ for some set $e$ with
      $\operatorname{Cap}_l(e)=0$;
  \item\label{cp:energy} $\displaystyle \|\nabla^l u_\Eb\|_{L^2(\R^\Nb)}^2\le C_0\,\operatorname{Cap}_l(\Eb)$;
  \item\label{cp:decay}  for every $x\notin\Eb$ and every $0\le j\le l$,
        \begin{equation}\label{eq:cp_decay}
           |\nabla^j u_\Eb(x)|\ \le\ C_0\,\operatorname{Cap}_l(\Eb)\,
           \operatorname{dist}(x,\Eb)^{\,2l-\Nb-j}.
        \end{equation}
    \item\label{cp:bdd} $0\le u_\Eb\le C_0$ on $\R^\Nb$.
\end{enumerate}
\end{proposition}

\begin{proof}
Let $\Eb\subset \mathbb{R}^\Nb$ be a compact set with $\operatorname{Cap}_l(\Eb) > 0$, in particular, due to 
Remark~\ref{rem:S_R_capacities}, $\operatorname{Cap}_l^R(\Eb) > 0$. By Theorem 2.2.7 in \cite{AdamsHedberg}, there exists a 
non-negative Radon measure $\nu_\Eb$ with $\operatorname{supp}\{\nu_\Eb\}\subset \Eb$ such that
\begin{equation}\label{eq:thm227}
    \nu_\Eb(\R^\Nb) = \operatorname{Cap}_l^R(\Eb) = \|I_l \nu_\Eb\|_{L^2(\R^\Nb)}^2
\end{equation}
and $u_\Eb:= I_l(I_l\nu_\Eb)$ satisfies 
\begin{equation*}
    u_\Eb \geq 1
    \qquad
    \operatorname{Cap}_l^R\text{-quasi-everywhere on }\Eb.
\end{equation*}
By Remark \ref{rem:S_R_capacities}, this gives \eqref{cp:one}.

Since $\nabla^l I_l$ is an isometry on $L^2(\mathbb{R}^\Nb)$, \eqref{eq:thm227} implies
\begin{equation*}
    \|\nabla^l u_\Eb\|_{L^2(\R^\Nb)}^2 = \|\nabla^l I_l(I_l\nu_\Eb)\|_{L^2(\R^\Nb)}^2  \simeq \|I_l\nu_\Eb\|_{L^2(\R^\Nb)}^2 = 
    \operatorname{Cap}_l^R(\Eb) \simeq \operatorname{Cap}_l(\Eb),
\end{equation*}
which is \eqref{cp:energy}; all  the involved constants  depend only on $\Nb$ and $l$.

For \eqref{cp:decay} we use the Riesz composition formula, see for example \cite[Chapter~5.1]{Stein},
\begin{equation*}
    \int_{\R^\Nb}|x-z|^{l-\Nb}|z-y|^{l-\Nb}\,\mathrm{d}z=c_{\Nb,l}\,|x-y|^{2l-\Nb},
\end{equation*}
which gives
\begin{equation}\label{eq:u_repr}
    u_\Eb(x)=c_{\Nb,l}\int_{\Eb}|x-y|^{2l-\Nb}\,\mathrm{d}\nu_\Eb(y).
\end{equation}
Then, for $x\notin\Eb$ with $r=\operatorname{dist}(x,\Eb)$, using $2l<\Nb$, we estimate
\begin{equation*}
    |\nabla^j u_\Eb(x)|
   =c_{\Nb,l}\Bigl|\int_{\Eb}\nabla_x^j|x-y|^{2l-\Nb}\,\mathrm{d}\nu_\Eb(y)\Bigr|
   \lesssim r^{\,2l-\Nb-j}\,\nu_\Eb(\R^\Nb)
   \lesssim \operatorname{Cap}_l(\Eb)\,r^{\,2l-\Nb-j}.
\end{equation*}
This gives \eqref{cp:decay}.

It remains to prove \eqref{cp:bdd}. Since $\nu_\Eb$ is a non-zero measure, the non-negativity of $u_\Eb$ follows from 
\eqref{eq:u_repr}.

For the upper bound, recall that Theorem 2.2.7 in \cite{AdamsHedberg} also
asserts that
\begin{equation}\label{eq:u_le_1_supp}
    u_\Eb\le 1 \qquad\text{on } \operatorname{supp}\nu_\Eb .
\end{equation}
Let now $x\in\R^\Nb$ be an arbitrary point. If
$x\in\operatorname{supp}\nu_\Eb$, then $u_\Eb(x)\le1$ by
\eqref{eq:u_le_1_supp}. Otherwise, we choose a nearest point
$x^*\in\operatorname{supp}\nu_\Eb$, so that
\begin{equation*}
    |x-x^*|=\operatorname{dist}(x,\operatorname{supp}\nu_\Eb)\le|x-y|
\end{equation*}
for every $y\in\operatorname{supp}\nu_\Eb$. Then, for all such $y$,
\[
|x^*-y|\ \le\ |x^*-x|+|x-y|\ \le\ 2\,|x-y|,
\]
and hence, since $2l-\Nb<0$,
\[
|x-y|^{2l-\Nb}\ \le\ 2^{\Nb-2l}\,|x^*-y|^{2l-\Nb}.
\]
Integrating this inequality against $d\nu_\Eb(y)$ in \eqref{eq:u_repr} and
using \eqref{eq:u_le_1_supp} at the point $x^*$, we obtain
\[
u_\Eb(x)\ \le\ 2^{\Nb-2l}\,u_\Eb(x^*)\ \le\ 2^{\Nb-2l}.
\]
This proves \eqref{cp:bdd}.
\end{proof}

For a Radon measure $\mu$ and a Borel set 
$\Eb\subset\mathbb{R}^\Nb$ with $\operatorname{Cap}_l(\Eb)>0$, we define
\begin{equation*}
    Z_\mu(\Eb) := \frac{\mu(\Eb)}{\operatorname{Cap}_l(\Eb)}.
\end{equation*}
This ratio, measuring the concentration of $\mu$ relative to capacity, plays a central role throughout the paper.

We consider the following condition which guarantees that the operator $H_{\mu}$, defined further on, is semibounded:

\begin{condition}\label{cond.cap}
    The measure $\mu$ satisfies
    \begin{equation}\label{2.1}
    \lim_{d \to 0}
    \sup_{x \in \mathbb{R}^\Nb}
    \sup_{\substack{\Eb \subset Q_d(x) \\ \operatorname{Cap}_l(\Eb)>0}}
    Z_\mu(\Eb)
    = 0.
\end{equation}
\end{condition}
Throughout the paper, this condition is always supposed to be fulfilled, unless the opposite is explicitly declared.   

masses.

The following lemma shows that Condition~\ref{cond.cap} forces the measure $\mu$ to be absolutely continuous with respect to the 
capacity
$\operatorname{Cap}_l$: sets of zero capacity carry no mass. In particular,
since the capacity of a single point is zero, $\mu$ possesses no point
masses.

\begin{lemma}\label{lem:mu_ll_cap}
Assume $\mu$ satisfies Condition~\ref{cond.cap}. If $K\subset\R^\Nb$ is a Borel
set with $\operatorname{Cap}_l(K)=0$, then $\mu(K)=0$.
\end{lemma}

\begin{proof}
By Condition~\ref{cond.cap}, we may fix $d>0$ so small that
\[
M(d):=\sup_{x\in\mathbb R^\Nb}\ 
\sup_{\substack{\Eb\subset Q_d(x)\\ \operatorname{Cap}_l(\Eb)>0}}
Z_\mu(\Eb)\le 1 .
\]
Cover $\mathbb R^\Nb$ by countably many open cubes of edgelength $d$, say
$\{Q_j\}_{j\in\mathbb N}\subset\mathbb Q$, and denote by $x_j$ the center of
$Q_j$, so that $Q_j=Q_d(x_j)$. Since 
\begin{equation*}
    \mu(K)\le\sum_j \mu(K\cap Q_j),
\end{equation*}
it suffices to show
that $\mu(K\cap Q_j)=0$ for every $j$.

Fix $j$ and set $\Eb=K\cap Q_j$. If $\Eb=\varnothing$ there is nothing to
prove, so assume $\Eb\ne\varnothing$. Since $K$ is Borel with
$\operatorname{Cap}_l(K)=0$, its outer capacity vanishes:
$\overline{\operatorname{Cap}}_l(K)=0$. Hence, given $\varepsilon>0$, there is
an open set $G\supset K$ with $\operatorname{Cap}_l(G)<\varepsilon$. Put
\[
G_\varepsilon:=G\cap Q_j .
\]
Then $G_\varepsilon$ is open, $\Eb\subset G_\varepsilon\subset Q_j$, and by
monotonicity of the capacity
\[
\operatorname{Cap}_l(G_\varepsilon)\le\operatorname{Cap}_l(G)<\varepsilon .
\]
Moreover, $G_\varepsilon$ is a nonempty open set, so it contains a ball
$B_r$ of some radius $r>0$, and since $2l<\Nb$ we have
\begin{equation*}
    \operatorname{Cap}_l(G_\varepsilon)\ge\operatorname{Cap}_l(\overline{B_{r/2}})
\simeq r^{\Nb-2l}>0.
\end{equation*}
Thus $\operatorname{Cap}_l(G_\varepsilon)>0$, and
$G_\varepsilon$ is an admissible set in \eqref{2.1}, which yields
\[
\mu(\Eb)\le\mu(G_\varepsilon)
= Z_\mu(G_\varepsilon)\,\operatorname{Cap}_l(G_\varepsilon)
\le M(d)\,\varepsilon\le\varepsilon .
\]
Letting $\varepsilon\to0$ gives $\mu(\Eb)=\mu(K\cap Q_j)=0$, and summing over
$j$ we conclude $\mu(K)=0$.
\end{proof}

We recall the following standard Ahlfors-type conditions, which will be used
in the comparison with known estimates.

\begin{definition}[Ahlfors conditions]\label{def:Ahlfors_conditions}
Let $\nu$ be a locally finite Borel measure on $\mathbb R^\Nb$, and let
$\Sigma=\operatorname{supp}\nu$.

We say that $\nu$ satisfies the upper Ahlfors condition of order $s>0$ if
there exists a constant $c_\nu>0$ such that
\begin{equation}\label{upper_Ahlfors_condition}
    \nu(Q_d(x))\leq c_\nu d^s,
    \qquad x\in\Sigma,\quad d>0.
\end{equation}

We say that $\nu$ satisfies the lower Ahlfors condition of order $s>0$ if
there exists a constant $c'_\nu>0$ such that
\begin{equation}\label{lower_Ahlfors_condition}
    \nu(Q_d(x))\geq c'_\nu d^s,
    \qquad x\in\Sigma,\quad 0<d\le\operatorname{diam}\Sigma.
\end{equation}

If both conditions hold, then $\nu$ is called $s$-Ahlfors regular.
\end{definition}

We shall use the following two consequences of the Ahlfors conditions when
comparing with known Lieb-Thirring estimates.

\begin{lemma}\label{l: measure_upper_bound_cap}
    Assume that measure $\nu$ satisfies the upper Ahlfors condition \eqref{upper_Ahlfors_condition} with $s > \Nb -2l$ and 
    \begin{equation*}
        \vartheta=\frac{s}{2l-\Nb+s}.
    \end{equation*}
    Then there exists a constant $C_{\Nb,l,c_\nu}>0$ such that 
    \begin{equation*}
        \nu(\Eb)^{1 - \frac{1}{\vartheta}} \leq C_{\Nb,l,\mathcal{A}_\nu} \; \operatorname{Cap}_l(\Eb)
    \end{equation*}
    for every Borel set $\Eb\subset \mathbb{R}^\Nb$.
\end{lemma}

\begin{proof}
    Let $q=s(\Nb/2-l)^{-1}>2$. By the D.~R.~Adams Theorem on Riesz potentials, see Theorem 2 in \cite[Section 1.4.1]{Mazya},
\begin{equation*}
    \|I_l f\|_{L^q(\mathbb{R}^\Nb; \nu)} \leq C_{\Nb,l,c_\nu}\|f\|_{L^2(\mathbb{R}^\Nb)}
\end{equation*}
for all $f\in L^2(\mathbb{R}^\Nb)$. Let $\Eb\subset \mathbb{R}^\Nb$ be a Borel set and $f\in L^2(\mathbb{R}^\Nb)$ be a non-negative 
function such that $I_l f\geq 1$ on $\Eb$. Then we can estimate 
\begin{equation*}
    \nu(\Eb)^{\frac{1}{q}} \leq \|I_l f\|_{L^q(\Eb;\nu)} \leq \|I_l f\|_{L^q(\mathbb{R}^\Nb; \nu)} \leq 
    C_{\Nb,l,c_\nu}\|f\|_{L^2(\mathbb{R}^\Nb)}.
\end{equation*}
Taking the infimum over all admissible functions $f$ in the definition of $\operatorname{Cap}_l^R(\Eb)$ and squaring both sides gives 

\begin{equation*}
    \nu(\Eb)^{\frac{2}{q}} \leq  C_{\Nb,l,\mathcal{A}_\nu}^2 \operatorname{Cap}_l^R(\Eb).
\end{equation*}
Noting that 
\begin{equation*}
    \frac{2}{q} = 2\left( \frac{\Nb/2 - l}{s} \right) = 1 - \frac{1}{\vartheta}
\end{equation*}
and recalling that $\operatorname{Cap}_l$ and $\operatorname{Cap}_l^R$ are equivalent, see Remark \ref{rem:S_R_capacities}, we 
complete the proof.
\end{proof}

For a locally finite Borel measure $\nu$ on $\mathbb R^\Nb$, set
$\Sigma=\operatorname{supp}\nu$. For an arbitrary set $\Eb\subset\Sigma$ and
$r>0$, we denote

\begin{equation}\label{def:B_sets}
    \Bc_{\mathbb R^\Nb}(\Eb,r)
    =
    \{x\in\mathbb R^\Nb:\operatorname{dist}_\infty(x,\Eb)<r\},
    \qquad
    \Bc_\Sigma(\Eb,r)
    =
    \{y\in\Sigma:\operatorname{dist}_\infty(y,\Eb)<r\},
\end{equation}
where
\begin{equation*}
    \|x\|_\infty=\max_{1\leq j\leq \Nb}|x_j|,
    \qquad
    \operatorname{dist}_\infty(x,\Eb)
    =
    \inf\{\|x-y\|_\infty:\ y\in \Eb\}.
\end{equation*}

\begin{lemma}\label{lem:tube_estimate}
Suppose that $\nu$ satisfies the lower Ahlfors condition of order $s>0$,
that is, \eqref{lower_Ahlfors_condition} holds. Then, for every set $\Eb\subset\Sigma$ and every $r<\operatorname{diam}(\Sigma)$,
\begin{equation}\label{eq:tube_estimate}
    \bigl|\Bc_{\mathbb R^\Nb}(\Eb,r)\bigr|
    \leq
    \frac{4^\Nb}{c_\nu'}\,r^{\Nb-s}\,
    \nu\bigl(\Bc_\Sigma(\Eb,r)\bigr).
\end{equation}
\end{lemma}

\begin{proof}
If $\Eb=\emptyset$, there is nothing to prove. Otherwise, by Zorn's lemma, there exists a
maximal $r$-separated family $\{y_i\}_{i\in I}\subset \Eb$, that is,
\begin{equation*}
    \|y_i-y_j\|_\infty\geq r,
    \qquad i\neq j,
\end{equation*}
and the family cannot be enlarged inside $\Eb$ while preserving this property.
Since $\mathbb R^\Nb$ is separable, the index set $I$ is at most countable.

The cubes $Q_r(y_i)$ are pairwise disjoint. Indeed, if
$z\in Q_r(y_i)\cap Q_r(y_j)$, then
\begin{equation*}
    \|y_i-y_j\|_\infty
    \leq
    \|y_i-z\|_\infty+\|z-y_j\|_\infty
    <
    \frac{r}{2}+\frac{r}{2}
    =
    r,
\end{equation*}
which contradicts the $r$-separation.

Since $y_i\in\Sigma$, the lower Ahlfors condition gives
\begin{equation*}
    \nu(Q_r(y_i))\geq c_\nu' r^s.
\end{equation*}
Moreover, $\Sigma=\operatorname{supp}\nu$ and
\begin{equation*}
    Q_r(y_i)\cap\Sigma\subset \Bc_\Sigma(\Eb,r).
\end{equation*}
Therefore, by the disjointness of the cubes, for every finite set
$J\subset I$,
\begin{equation*}
    c_\nu r^s |J|
    \leq
    \sum_{i\in J}\nu(Q_r(y_i))
    =
    \sum_{i\in J}\nu(Q_r(y_i)\cap\Sigma)
    \leq
    \nu\bigl(\Bc_\Sigma(\Eb,r)\bigr),
\end{equation*}
where $|J|$ denotes the cardinality of $J$. Taking the supremum over all finite sets $J\subset I$, we obtain
\begin{equation}\label{eq:tube_count}
    \,c_\nu r^s|I|
    \leq
    \nu\bigl(\Bc_\Sigma(\Eb,r)\bigr).
\end{equation}

Let $x\in \Bc_{\mathbb R^\Nb}(\Eb,r)$.
Then there exists $z\in \Eb$ such that
\begin{equation*}
    \|x-z\|_\infty<r.
\end{equation*}
By maximality of the family $\{y_i\}_{i\in I}$, there exists $i\in I$ such
that
\begin{equation*}
    \|z-y_i\|_\infty<r.
\end{equation*}
Hence
\begin{equation*}
    \|x-y_i\|_\infty<2r,
\end{equation*}
and therefore $x\in Q_{4r}(y_i)$. Thus
\begin{equation*}
    \Bc_{\mathbb R^\Nb}(\Eb,r)
    \subset
    \bigcup_{i\in I} Q_{4r}(y_i).
\end{equation*}
Combining this inclusion with \eqref{eq:tube_count}, we get
\begin{equation*}
    \bigl|\Bc_{\mathbb R^\Nb}(\Eb,r)\bigr|
    \leq
    \,|Q_{4r}||I|
    =
    \,(4r)^\Nb|I|
    \leq
    \frac{4^\Nb}{c_\nu}\,r^{\Nb-s}\,
    \nu\bigl(\Bc_\Sigma(\Eb,r)\bigr).
\end{equation*}
This completes the proof.
\end{proof}

\subsection{Formulation of the problem}
In this paper, we consider operators acting formally as
\begin{equation*}
    (-\Delta)^l - \mu,
\end{equation*}
with $\mu$ playing the role of potential. With this expression we associate the Schr{\"o}dinger operator $H_{\mu}$ defined by means 
of the quadratic form
\begin{equation}\label{quad_form_LO}
    a_\mu[u] = \int_{\mathbb{R}^\Nb} |\nabla^l u|^2 \mathrm{d}x - \int_{\mathbb{R}^\Nb} |u|^2 \mathrm{d}\mu
\end{equation}
and $\mu$ will be called the measure-potential.
\begin{proposition}\label{Prop.form}
  Under Condition \ref{cond.cap}, the quadratic form  \eqref{quad_form_LO} is correctly 
  defined, semibounded and closable on  $\Ht^l(\R^\Nb)$ and thus defines a self-adjoint operator.
\end{proposition}

The proof goes in the standard way in several steps.  Namely $a_\mu[u]$ is considered first on functions $u\in 
C_0^{\infty}(\R^{\Nb})$ and then 
extended by continuity to the whole $\Ht^l(\R^\Nb)$. By \cite[Theorem 11.3]{Mazya},  
 \begin{equation}\label{Maz_N}
        \int_{\mathbb{R}^\Nb} |u|^2 \mathrm{d}\mu \leq C_{\Nb,l} \sup_{\Eb\subset \mathbb{R}^\Nb} Z_\mu(\Eb)  \|\nabla^l u 
        \|_{L^2(\mathbb{R}^\Nb)}^2
    \end{equation}
    for  $u\in C_0^{\infty}(\R^{\Nb})$ with some constant $C_{\Nb,l}>0$ not depending on $u$ and $\mu$;  by continuity, \eqref{Maz_N} 
    extends to $u\in\htt^l(\R^{\Nb})$.

A consequence of \eqref{Maz_N} is the following local estimate:
\begin{lemma}[{\cite[Theorem 11.3]{Mazya}}]\label{Mazya_thm_11.3}
   For a \emph{unit} cube $Q\subset \mathbb{R}^\Nb$, the inequality holds
    \begin{equation*}
        \int_{Q} |u|^2 \mathrm{d}\mu \leq C_{\Nb,l} \sup_{\Eb\subset Q} Z_\mu(\Eb)  \left( \|\nabla^l u \|_{L^2(Q)}^2 + 
        \|u\|_{L^2(Q)}^2 \right)
    \end{equation*}
for all $u\in \Ht^l(Q)$ and for some constant $C_{\Nb,l}>0$ not depending on $u$ and $\mu$. 
\end{lemma}
We give a simple proof
\begin{proof}
 By the Stein extension theorem \cite[Chapter~6]{Stein} (see also the precise formulation in Theorem~5.24 of \cite{AdamsFournier}), 
 there exists a bounded  extension operator $ \Ht^l(Q)\ni u\mapsto \Tilde{u}\in \Ht^l(\R^{\Nb})$ of the function $u$ to 
 $\mathbb{R}^\Nb$ such that 
    \begin{equation*}
        \|\Tilde{u}\|_{\Ht^l(\mathbb{R}^\Nb)} \leq C_{\Nb,l} \|u\|_{\Ht^l(Q)}.
    \end{equation*}
    Since $2l<\Nb$, we have $\Ht^l(\R^\Nb)\hookrightarrow\htt^l(\R^\Nb)$, so in
    particular $\Tilde{u}\in\htt^l(\R^\Nb)$. Applying  inequality \eqref{Maz_N} to the restricted measure
    $\mu\arrowvert_{Q}$ and the extension $\Tilde{u}$, we obtain
    \begin{align*}
        \int_{Q} |u|^2 \mathrm{d}\mu = \int_{Q} |\Tilde{u}|^2 \mathrm{d}\mu\arrowvert_{Q} & \leq C_{\Nb,l} \sup_{\Eb\subset Q} 
        Z_\mu(\Eb)   \|\nabla^l 
        \Tilde{u} \|_{L^2(\mathbb{R}^\Nb)}^2 \\
        &\leq  C_{\Nb,l} \sup_{\Eb\subset Q} Z_\mu(\Eb) \left( \|\nabla^l u \|_{L^2(Q)}^2 + \|u\|_{L^2(Q)}^2 \right).
    \end{align*}
    This completes the proof.
\end{proof}

Next, we rescale the above inequality to an arbitrary cube of sidelength $d > 0$ and locate the dependence of the constants in the 
estimate on $d$:

\begin{lemma}\label{lem:local_estimate_scaled}
    Let $\mu$ be a Radon measure on $\mathbb R^\Nb$. Let $Q_d\subset\mathbb R^\Nb$
    be a cube of sidelength $d>0$. Then
    \begin{equation}\label{Lem_2.5}
        \int_{Q_d} |u|^2\,\mathrm{d}\mu
        \leq
        \varpi
        \sup_{\substack{\Eb\subset Q_d}}
        Z_\mu(\Eb)
        \left(
            \|\nabla^l u\|_{L^2(Q_d)}^2
            +
            d^{-2l}\|u\|_{L^2(Q_d)}^2
        \right)
    \end{equation}
    for all $u\in \Ht^l(Q_d)$, where $\varpi>0$ depends only on
    $\Nb$ and $l$.
\end{lemma}

\begin{proof}
    By translation, it is enough to prove the estimate for
    $Q_d=(0,d)^\Nb$. Let $Q_1=(0,1)^\Nb$ and set
    \begin{equation*}
        v(y)=u(dy), \qquad y\in Q_1.
    \end{equation*}
    Define the rescaled measure $\widetilde\mu$ on $Q_1$ by
    \begin{equation*}
        \widetilde\mu(A)=\mu(dA),
        \qquad A\subset Q_1.
    \end{equation*}
    Then
    \begin{equation*}
        \int_{Q_d}|u|^2\,\mathrm{d}\mu
        =
        \int_{Q_1}|v|^2\,\mathrm{d}\widetilde\mu.
    \end{equation*}

    Applying Lemma~\ref{Mazya_thm_11.3} on $Q_1$, we obtain
    \begin{equation*}
        \int_{Q_1}|v|^2\,\mathrm{d}\widetilde\mu
        \leq
        C_{\Nb,l}
        \sup_{\substack{A\subset Q_1}}
        Z_{\widetilde\mu}(A)
        \left(
            \|\nabla^l v\|_{L^2(Q_1)}^2
            +
            \|v\|_{L^2(Q_1)}^2
        \right).
    \end{equation*}

    The change of variables $x=dy$ gives
    \begin{equation*}
        \|v\|_{L^2(Q_1)}^2
        =
        d^{-\Nb}\|u\|_{L^2(Q_d)}^2
    \end{equation*}
    and
    \begin{equation*}
        \|\nabla^l v\|_{L^2(Q_1)}^2
        =
        d^{2l-\Nb}\|\nabla^l u\|_{L^2(Q_d)}^2.
    \end{equation*}

    If $A\subset Q_1$ and $\Eb=dA\subset Q_d$, then $\widetilde\mu(A)=\mu(\Eb)$. Moreover, by the homogeneity of the Sobolev 
    capacity,
    \begin{equation*}
        \operatorname{Cap}_l(\Eb)
        =
        d^{\Nb-2l}\operatorname{Cap}_l(A).
    \end{equation*}
    Hence
    \begin{equation*}
        Z_{\widetilde\mu}(A)
        =
        \frac{\widetilde\mu(A)}{\operatorname{Cap}_l(A)}
        =
        d^{\Nb-2l}
        \frac{\mu(\Eb)}{\operatorname{Cap}_l(\Eb)}
        =
        d^{\Nb-2l}Z_\mu(\Eb).
    \end{equation*}
    Therefore,
    \begin{equation*}
        \sup_{\substack{A\subset Q_1}}
        Z_{\widetilde\mu}(A)
        =
        d^{\Nb-2l}
        \sup_{\substack{\Eb\subset Q_d}}
        Z_\mu(\Eb).
    \end{equation*}

    Combining these identities, we get
    \begin{align*}
        \int_{Q_d}|u|^2\,\mathrm{d}\mu
        &\leq
        C_{\Nb,l} d^{\Nb-2l}
        \sup_{\substack{\Eb\subset Q_d}}
        Z_\mu(\Eb)
        \left(
            d^{2l-\Nb}\|\nabla^l u\|_{L^2(Q_d)}^2
            +
            d^{-\Nb}\|u\|_{L^2(Q_d)}^2
        \right) \\
        &=
        C_{\Nb,l}
        \sup_{\substack{\Eb\subset Q_d}}
        Z_\mu(\Eb)
        \left(
            \|\nabla^l u\|_{L^2(Q_d)}^2
            +
            d^{-2l}\|u\|_{L^2(Q_d)}^2
        \right).
    \end{align*}
    This proves the lemma.
\end{proof}

Next, we recall the following result:

\begin{proposition}[\cite{RozenblumShargorodsky}, Theorem 4.3]\label{coor_syst}
    Let $\mu$ be a Radon measure on $\mathbb{R}^\Nb$ without point masses (this holds for sure if Condition \ref{cond.cap} is met). 
    Then there exists an orthogonal coordinate system with the following property: if $\mathbb{Q}$ denotes the collection of all open 
    cubes in $\mathbb{R}^\Nb$ whose edges are parallel to the coordinate axes, then
    \begin{equation*}
        \mu(\partial Q) = 0 \qquad \text{for all } Q \in \mathbb{Q}.
    \end{equation*}
    In particular, for any $x \in \mathbb{R}^\Nb$, the function $g(d) = \mu(Q_d(x))$ is continuous in $d \ge 0$.
\end{proposition}

\begin{remark}\label{rem_syst}
    From now on, we always assume that the coordinate system is the one described in Proposition~\ref{coor_syst}. In particular, now 
    we do not need to distinguish between open and closed cubes when integrating against the measure $\mu$.
\end{remark}

Now, we are ready to prove Proposition \ref{Prop.form}.

\begin{proof}
    It is clear that the form $a_{\mu}$ is densely defined. Let us show that it is closed and lower semibounded.

   Fix $d>0$. For $\mathbf{k} = (k_1, \dots, k_\Nb) \in \mathbb{Z}^\Nb$, define
    \begin{equation*}
        Q_{\mathbf{k}} := \prod_{j=1}^\Nb \left( d k_j,\, d(k_j + 1) \right).
    \end{equation*}
    Then, the collection $\{ \overline{Q_{\mathbf{k}}} \}_{\mathbf{k} \in \mathbb{Z}^\Nb}$ forms a cover of $\mathbb{R}^\Nb$ by cubes 
    of edgelength $d$. Recalling Proposition \ref{coor_syst} and Remark \ref{rem_syst}, we write 
    \begin{equation*}
        \int_{\mathbb{R}^\Nb} |u|^2 \mathrm{d}\mu = \sum_{\mathbf{k} \in \mathbb{Z}^\Nb} \int_{Q_{\mathbf{k}}} |u|^2 \mathrm{d}\mu.
    \end{equation*}
    Therefore, inequality \eqref{Lem_2.5}, applied to each cube, gives
    \begin{align*}
        \int_{\mathbb{R}^\Nb} |u|^2 \mathrm{d}\mu &\leq \sum_{\mathbf{k} \in \mathbb{Z}^\Nb} \varpi \sup_{\Eb\subset Q_{\mathbf{k}}} 
        Z_\mu(\Eb)
        \left( \|\nabla^l u\|_{L^2(Q_{\mathbf{k}})}^2 + d^{-2l} \|u\|_{L^2(Q_{\mathbf{k}})}^2 \right)\\
        & \leq \varpi \sup_{\mathbf{k}\in \mathbb{Z}^\Nb}\sup_{\Eb\subset Q_{\mathbf{k}}} Z_\mu(\Eb)
        \left( \|\nabla^l u\|_{L^2(\mathbb{R}^\Nb)}^2 + d^{-2l} \|u\|_{L^2(\mathbb{R}^\Nb)}^2 \right).
    \end{align*} 
    Since $\mu$ satisfies Condition~\ref{cond.cap}, there exists $d>0$
    sufficiently small such that
    \begin{equation*}
        \varpi \sup_{\mathbf{k}\in \mathbb{Z}^\Nb}\sup_{\Eb\subset Q_{\mathbf{k}}}
       Z_\mu(\Eb) < 1.
    \end{equation*}
    Now, the KLMN theorem \cite[Theorem X.17]{ReedSimon2} implies that $a_{\mu}$ is closed and lower semibounded.
\end{proof}

Thus, we see that the quadratic form $a_\mu$ defines, via the
representation theorem, a self-adjoint operator $H_\mu$. The main objective of this paper is to study the negative spectrum of the 
operator $H_\mu$.

We will also need to consider the quadratic form $a_{\mu}$ restricted
to sub-domains of $\mathbb{R}^\Nb$. Specifically, we consider domains
of the form $U = Q$ for some cube $Q \in \mathbb{Q}$, or
$U = \mathbb{R}^\Nb_+$, where
\begin{equation*}
    \mathbb{R}^\Nb_+ = \{x \in \mathbb{R}^\Nb : x_1 > 0\}.
\end{equation*}
We introduce the quadratic form
\begin{equation}\label{quad_form_LO_loc}
    a_{\mu,U}[u] = \int_{U} |\nabla^l u|^2 \mathrm{d}x - \int_{U} |u|^2 \mathrm{d}\mu
\end{equation}
with the domain $D(a_{\mu,U}) = \Ht^l(U)$. 
\begin{proposition}\label{Prop.form.loc}
  Suppose that $\mu$ is a Radon measure satisfying Condition \ref{cond.cap}. Let $U$ be either a cube $Q \in \mathbb{Q}$ or the 
  half-space $\mathbb{R}^{\Nb}_+$. Then, the quadratic form  \eqref{quad_form_LO_loc} is correctly defined, semibounded and closable 
  on $\Ht^l(U)$ and thus defines a self-adjoint operator. Classically, this definition produces the operator with Neumann boundary 
  conditions.
\end{proposition}
\begin{proof}
    The proof is the same as in Proposition \ref{Prop.form}: we decompose $U$ into sufficiently small cubes and then apply 
    Condition~\ref{cond.cap} together with the KLMN theorem.
\end{proof}
Finally, we recall the following Poincar\'e inequality; see for example Lemma 1.1.11 in \cite{Mazya}. The constant $\rho$ appearing 
below will play a role in the spectral estimates of subsequent sections.

\begin{lemma}\label{l_Poincare}
    For any $l$, $\Nb\in\mathbb{N}$ (\emph{here}, not necessarily $2l<\Nb$) and  any cube $Q$ of edgelength $d$,  there exists a 
    constant $\rho = \rho(\Nb,l)>0$ such that 
    \begin{equation*}
        \frac{1}{d^{2l}}\|f\|_{L^2(Q)}^2 \leq \rho \|\nabla^l f\|_{L^2(Q)}^2
    \end{equation*}
    for all $f \in \mathbb{P}_{l-1}(\mathbb{R}^\Nb)^{\perp}\cap \Ht^l(Q)$ (in $L^2(Q)$ sense), where 
    $\mathbb{P}_{l-1}(\mathbb{R}^\Nb)$ is the set of all polynomials of degree less than $l$. 
\end{lemma}

\section{The Otelbaev function. Definition and main properties}\label{sec:otelbaev}
In this section we introduce the Otelbaev function and establish some of its important properties. This construction is a 
modification of the Otelbaev 
function $q^*$ in the regime $2l<\Nb$, introduced in  \cite{O81}, \cite{O83}. In these papers, this function was used to study the 
spectrum of Schr{\"o}dinger-like operators with potential being a function, tending, in a certain averaged sense, to $-\infty$ at 
infinity, 
so that the whole spectrum is discrete. In our setting, Otelbaev function quantifies the qualitative Molchanov criterion for 
discreteness of the 
spectrum for $l=1$, see \cite{Molchanov} and also \cite{MazyaShubin}. 

The case of one-dimensional Schr\"odinger operators  was considered in \cite{Nursultanov16,FulscheNursultanov22}, where an
Otelbaev-type function was adapted to negative  measure-potentials in dimension one.

In the present paper, we adapt the definition of the Otelbaev function to
non-negative Radon measures, possibly singular with respect to the Lebesgue
measure, in the regime $2l<\Nb$. Moreover, we study the negative spectrum of
the operator, with the positive half-axis filled with the essential spectrum,
and this leads to certain changes.

 We begin with the definition. The Otelbaev function will be parametrized by a positive number $\tau$; this is reflected in the 
 notation.
\begin{definition}\label{def:delta}
   For a fixed $\tau>0$, we define 
    \begin{equation*}
        \delta_{\mu, \tau}(x) = \sup \left\{ d: \; \sup_{\Eb\subset Q_d(x)} Z_\mu(\Eb) < \frac{1}{\tau} \right\},
    \end{equation*}
    where the supremum is taken over Borel sets $\Eb$ with non-zero capacity. We define the Otelbaev function as follows 
    \begin{equation*}
        q_{\mu,\tau}^*(x) = \frac{1}{\delta_{\mu,\tau}^{2l}(x)}.
    \end{equation*}
    If $\delta_{\mu,\tau}(x) = \infty$, we set $q_{\mu,\tau}^*(x) = 0$. We will also use the notation
    \begin{equation*}
        Q_{\mu,\tau}(x):= Q_{\delta_{\mu,\tau}(x)}(x).
    \end{equation*}
\end{definition}

Due to Condition \ref{cond.cap}, the following inequality holds 
\begin{equation*}
    \sup_{\Eb\subset Q_d(x)} Z_\mu(\Eb) < \frac{1}{\tau}
\end{equation*}
for any $\tau>0$, $x\in \mathbb{R}^\Nb$ and sufficiently small $d>0$. Therefore, the Otelbaev function is well defined.

\begin{remark}\label{rem:infinite_d} 
If, for some $\tau>0,$ there exists $x_0\in\mathbb R^\Nb$ such that $\delta_{\mu,\tau}(x_0)=\infty$, then 
$\delta_{\mu,\tau}(x)=\infty$ for all $x\in\mathbb{R}^\Nb$. Indeed, in this case, $Z_\mu(\Eb) < 1/\tau$ for all bounded Borel  sets 
$\Eb\subset \mathbb{R}^\Nb$, and hence, $\delta_{\mu,\tau}\equiv\infty$. 
\end{remark}

\begin{remark}[Monotonicity]
If $\mu_1\le\mu_2$, then $Z_{\mu_1}(\Eb)\le Z_{\mu_2}(\Eb)$ for every Borel
$\Eb$ with $\operatorname{Cap}_l(\Eb)>0$, hence
$\delta_{\mu_1,\tau}\ge\delta_{\mu_2,\tau}$ and
$q^*_{\mu_1,\tau}\le q^*_{\mu_2,\tau}$. In particular $0<\eta_1\le\eta_2$ gives
$q^*_{\eta_1\mu,\tau}\le q^*_{\eta_2\mu,\tau}$. The Otelbaev function is also monotone in the variable $\tau$: for 
$0<\tau_1\le\tau_2$, $q^*_{\mu,\tau_1}\le q^*_{\mu,\tau_2}$.
\end{remark}

\begin{remark}[Scaling properties]\label{rem:qstar_scaling}
We first note the important coupling-scaling property, $ q^*_{\eta\mu,\tau}=q^*_{\mu,\eta\tau}$ for $\eta>0$. We also record the 
behaviour under spatial scaling. For $t>0$, we define
\begin{equation*}
    \widetilde\mu_t(\Eb):=\mu(t\Eb),
    \qquad
    \mu_t(\Eb):=t^{2l-\Nb}\mu(t\Eb),
    \qquad \Eb\subset\mathbb R^\Nb.
\end{equation*}
By the homogeneity of the capacity,  $Z_{\mu_t}(\Eb)=Z_\mu(t\Eb)$. Consequently,
\begin{equation*}
    \delta_{\mu_t,\tau}(x)=t^{-1}\delta_{\mu,\tau}(tx),
    \qquad
    q^*_{\mu_t,\tau}(x)=t^{2l}q^*_{\mu,\tau}(tx).
\end{equation*}
Due to coupling-scaling property, 
\begin{equation*}
    \delta_{\widetilde\mu_t,\tau}(x)
    =
    t^{-1}\delta_{\mu,t^{\Nb-2l}\tau}(tx),
    \qquad
    q^*_{\widetilde\mu_t,\tau}(x)
    =
    t^{2l}q^*_{\mu,t^{\Nb-2l}\tau}(tx).
\end{equation*}
\end{remark}

\begin{lemma}[Lipschitz continuity]\label{lem:continuity_d}
Let $\tau>0$. Suppose that $\delta_{\mu,\tau}(x)$ is finite at some point
$x\in\mathbb R^\Nb$; equivalently, by Remark~\ref{rem:infinite_d},
$\delta_{\mu,\tau}$ is finite everywhere. Then $\delta_{\mu,\tau}$ is
Lipschitz continuous on $\mathbb R^\Nb$ and satisfies
\begin{equation*}
    |\delta_{\mu,\tau}(x)-\delta_{\mu,\tau}(y)|
    \leq 2\|x-y\|_\infty,
    \qquad x,y\in\mathbb R^\Nb.
\end{equation*}
\end{lemma}

\begin{proof}
Let $x,y\in\mathbb R^\Nb$ and set $r=\|x-y\|_\infty$. Then,
\begin{equation*}
    Q_d(y)\subset Q_{d+2r}(x),
    \qquad
    Q_d(x)\subset Q_{d+2r}(y),
\end{equation*}
for every $d>0$. Consequently,
\begin{equation*}
    \sup_{\Eb\subset Q_d(y)}Z_\mu(\Eb)
    \leq
    \sup_{\Eb\subset Q_{d+2r}(x)}Z_\mu(\Eb),
\end{equation*}
and the analogous inequality with $x$ and $y$ interchanged.

Let $d< \delta_{\mu,\tau}(x)-2r$. Then $d+2r<\delta_{\mu,\tau}(x)$, and by the definition of
$\delta_{\mu,\tau}(x)$,
\[
\sup_{\Eb\subset Q_{d+2r}(x)}Z_\mu(\Eb)<\frac{1}{\tau}.
\]
Hence, by the inclusions above,
\[
\sup_{\Eb\subset Q_d(y)}Z_\mu(\Eb)<\frac{1}{\tau},
\]
which implies $d\le \delta_{\mu,\tau}(y)$. Since this is true for arbitrary $d< \delta_{\mu,\tau}(x)-2r$, we obtain
\[
\delta_{\mu,\tau}(y)\ge \delta_{\mu,\tau}(x)-2r.
\]
By symmetry,
\[
\delta_{\mu,\tau}(x)\ge \delta_{\mu,\tau}(y)-2r.
\]
Therefore,
\[
|\delta_{\mu,\tau}(x)-\delta_{\mu,\tau}(y)|\le 2\|x-y\|_\infty,
\]
so $\delta_{\mu,\tau}(\cdot)$ is Lipschitz continuous on
$\mathbb R^\Nb$.
\end{proof}

\begin{corollary}\label{cor:continuity_qstar}
For a fixed $\tau>0$,
the Otelbaev function $q_{\mu,\tau}^*(\cdot)$ is continuous on $\mathbb R^\Nb$.
\end{corollary}

\begin{proof}
If $\delta_{\mu,\tau}(x_0)=\infty$ for some point $x_0$, then by
Remark~\ref{rem:infinite_d} $\delta_{\mu,\tau}$ is infinite  on
$\mathbb R^\Nb$, and hence $q_{\mu,\tau}^*\equiv 0$, which is trivially continuous. Otherwise, $\delta_{\mu,\tau}(x)<\infty$ for all 
$x\in\mathbb R^\Nb$, and the statement follows from Lemma~\ref{lem:continuity_d}.
\end{proof}

\begin{remark}
    In contrast to the Otelbaev function in \cite[Definition 3.1]{FulscheNursultanovRozenblum2025}, $q_{\mu,\tau}^*(x)$ is not 
    necessarily continuous with respect to the parameter $\tau$; see the example  below.
\end{remark}

\begin{example}
   There exists a Radon measure $\mu$ on $\mathbb{R}^\Nb$ satisfying Condition \ref{cond.cap} such that the corresponding Otelbaev 
   function $q_{\mu,\tau}^*$ is not continuous with respect to $\tau$.

    Indeed, let us choose $\mu = \chi_{B_1(0)} \mathcal{L}$, where $\mathcal{L}$ is the Lebesgue measure on $\mathbb{R}^\Nb$, 
    $B_1(0)$ is the unit ball centered at the origin and $\chi_{B_1(0)}$ is the corresponding indicator function. For a  fixed 
    $x_0\in \mathbb{R}^\Nb$, $\|x_0\|_\infty>1$ we consider the sets
    \begin{equation*}
        \mathcal{X} = \{\tau>0: \; \delta_{\mu,\tau}(x_0) < \infty\},
        \qquad
        \mathcal{Y} = \{\tau>0: \; \delta_{\mu,\tau}(x_0) = \infty\}.
    \end{equation*}
    For $\tau>0$ large enough, we have
    \begin{equation*}
         Z_\mu(B_1(0))> \frac{1}{\tau},
    \end{equation*}
    therefore, the set  $\mathcal{X}$ is not empty. Moreover, for $\tau\in \mathcal{X}$,
    \begin{equation*}
        2\|x_0\|_\infty - 2 \leq \delta_{\mu,\tau}(x_0) \leq 2\|x_0\|_\infty + 2,
    \end{equation*}
    and consequently, 
    \begin{equation*}
        \frac{1}{(2\|x_0\|_\infty + 2)^{2l}} \leq q_{\mu,\tau}^*(x_0) \leq \frac{1}{(2\|x_0\|_\infty - 2)^{2l}}.
    \end{equation*}

    Next, let us show that $\mathcal{Y}$ is not empty. Clearly, for any $x\in\mathbb{R}^\Nb$ and $d>0$,
    \begin{equation}\label{eq_rest_L_satis_Cond_cap}
        \mu(Q_d(x)) = \mathcal{L} (Q_d(x) \cap B_1(0)) \leq \mathcal{L}(Q_d(x)) \leq C_{\Nb,l} d^\Nb.
    \end{equation}
     By Lemma~\ref{l: measure_upper_bound_cap}, applied with $s = \Nb$, the last estimate implies that 
    \begin{equation*}
        \mu(\Eb)^{1 - \frac{2l}{\Nb}} \leq C_{\Nb,l} \operatorname{Cap}_l(\Eb)
    \end{equation*}
    for any Borel set $\Eb$. Therefore, 
    \begin{equation}\label{est_quant_L}
        Z_\mu(\Eb) < C_{\Nb,l} \; \mu(\Eb)^{\frac{2l}{\Nb}} = C_{\Nb,l} \; \mathcal{L}(\Eb\cap B_1(0))^{\frac{2l}{\Nb}} < C_{\Nb,l}
    \end{equation}
    for any Borel set $\Eb$ and some constant $C_{\Nb,l}>0$ depending only on $\Nb$ and $l$. Therefore, for $\tau<1/C_{\Nb,l}$, it 
    follows that $\delta_{\mu,\tau}(x_0) = \infty$, and hence, $\mathcal{Y}$ is not empty.

    We summarize
    \begin{equation*}
        q_{\mu,\tau}^*
        \begin{cases}
            = 0, & \tau\in \mathcal{Y} \neq \emptyset,\\
            \geq \frac{1}{(2\|x_0\|_\infty + 2)^{2l}}, & \tau\in \mathcal{X }\neq \emptyset.
        \end{cases}
    \end{equation*}
    Since $(0,\infty) = \mathcal{X}\bigcup \mathcal{Y}$, we conclude that $q_{\mu,\tau}^*(x_0)$ is not continuous with respect to 
    $\tau$.
\end{example}

We will now show that the Otelbaev function has controlled local  variation.
\begin{lemma}\label{l_cont_var_LO}
 Suppose that $\mu$ is a Radon measure satisfying Condition \ref{cond.cap}. Fix some $\tau>0$ and $x\in\mathbb R^\Nb$. Then for every 
 pair of points $y,z$,
\begin{equation*}
    y \in Q_{\mu,\tau}(x)
    \qquad
    \text{and}
    \qquad
    z \in \frac{1}{2}Q_{\mu,\tau}(x),
\end{equation*}
one has
\begin{equation*}
    \frac{1}{2}\delta_{\mu,\tau}(y) \leq  \delta_{\mu,\tau}(x) \leq 2 \delta_{\mu,\tau}(z),
\end{equation*}
or equivalently,
\begin{equation*}
    \frac{1}{2^{2l}} q_{\mu,\tau}^*(z) \leq q_{\mu,\tau}^*(x) \leq 2^{2l} \, q_{\mu,\tau}^*(y).
\end{equation*}

\end{lemma}
\begin{proof}
Let $x \in \mathbb{R}^\Nb$. If $\delta_{\mu,\tau}(x) = \infty$, the claim is trivial. If $\delta_{\mu,\tau}(x)<\infty$, let $y 
\in Q_{\mu,\tau}(x)$. Then $Q_{\mu,\tau}(x) \subset Q_{2\delta_{\mu,\tau}(x)}(y)$. In particular, since our 
cubes are open, there exists $d'> \delta_{\mu,\tau}(x)$ such that 
\begin{equation*}
    Q_{\mu,\tau}(x) \subset Q_{d'}(x) \subset Q_{2\delta_{\mu,\tau}(x)}(y).
\end{equation*}
By the definition of $\delta_{\mu,\tau}(x)$, there exist sets $\Eb_k \subset Q_{d'}(x) \subset Q_{2\delta_{\mu,\tau}(x)}(y)$ such 
that
\begin{equation}\label{eq:E_k}
    Z_\mu(\Eb_k)\geq \frac{1}{\tau} - \frac{1}{k}.
\end{equation}
Suppose that $\delta_{\mu,\tau}(y) > 2\, \delta_{\mu,\tau}(x)$. Let $d'' \in (2\, \delta_{\mu,\tau}(x), \delta_{\mu,\tau}(y))$. Then 
$\Eb_k \subset Q_{d''}(y)$. Due to \eqref{eq:E_k},
\begin{equation*}
    \sup_{\Eb\subset Q_{d''}(y)}Z_\mu(\Eb) \geq \frac{1}{\tau} -\frac{1}{k}
\end{equation*}
for all $k\in \mathbb{N}$, and hence,
\begin{equation*}
    \sup_{\Eb\subset Q_{d''}(y)}Z_\mu(\Eb) \geq \frac{1}{\tau}.
\end{equation*}
Since $d'' < \delta_{\mu,\tau}(y)$, this contradicts the definition of $\delta_{\mu,\tau}(y)$. Thus, $\delta_{\mu,\tau}(y)\le 
2\,\delta_{\mu,\tau}(x)$.

Next, let $z \in \frac{1}{2}Q_{\mu,\tau}(x)$. Suppose that $\delta_{\mu,\tau}(z) \leq \delta_{\mu,\tau}(x)/2$. We will show that this 
is wrong. Since 
$\frac{1}{2}Q_{\mu,\tau}(x)$ is an open cube, we have $\overline{Q_{\mu,\tau}(z)} \subset Q_{\mu,\tau}(x)$. Hence, there exists 
$r'>\delta_{\mu,\tau}(z)$ and $r'' < \delta_{\mu,\tau}(x)$ such that
\begin{equation}\label{incl:bound_var}
    Q_{r'}(z) \subset Q_{r''}(x).
\end{equation}

Since $r'>\delta_{\mu,\tau}(z)$, by the definition of $\delta_{\mu,\tau}(z)$ there exist sets $\Eb_k\subset Q_{r'}(z)$ such that
\begin{equation*}
    Z_\mu(\Eb_k)\geq \frac{1}{\tau}-\frac{1}{k}.
\end{equation*}
By \eqref{incl:bound_var}, $\Eb_k\subset Q_{r''}(x)$, and consequently
\begin{equation*}
    \sup_{\Eb\subset Q_{r''}(x)} Z_\mu(\Eb)\geq \frac{1}{\tau}-\frac{1}{k}
\end{equation*}
for every $k\in\mathbb{N}$. Letting $k\to\infty$ gives
\begin{equation*}
    \sup_{\Eb\subset Q_{r''}(x)} Z_\mu(\Eb)\geq \frac{1}{\tau}.
\end{equation*}
Since $r''<\delta_{\mu,\tau}(x)$, this contradicts the definition of $\delta_{\mu,\tau}(x)$. Therefore, the opposite assumption 
holds, $\delta_{\mu,\tau}(z) > 
\delta_{\mu,\tau}(x)/2$. This completes the proof.
\end{proof}

\begin{lemma}\label{est_for_small_cubes}
    Let $\tau > 0$ and $\mu$ be a Radon measure satisfying Condition \ref{cond.cap}. Then for any $x \in \mathbb{R}^{\Nb}$ such that 
    $\delta_{\mu,\tau}(x)<\infty$,
    \begin{equation*}
        \sup_{\Eb\subset Q_{\mu,\tau}(x)} Z_\mu(\Eb) \leq \frac{1}{\tau}.
    \end{equation*}
\end{lemma}

\begin{proof}
    We prove the statement by contradiction. Assume there exists $\Eb\subset Q_{\mu,\tau}(x)$ such that $Z_\mu(\Eb) > 1/\tau$. By the 
    inner regularity of $\mu$, taking into account Proposition \ref{coor_syst} and Remark \ref{rem_syst}, we obtain
    \begin{equation*}
        \mu\left(Q_d(x)\cap \Eb\right) \rightarrow \mu(Q_{\mu,\tau}(x) \cap \Eb) = \mu(\Eb),
    \end{equation*}
    as $d\uparrow \delta_{\mu,\tau}(x)$. Due to the monotonicity of $\operatorname{Cap}_l(\cdot)$, for any $d < 
    \delta_{\mu,\tau}(x)$,
    \begin{equation*}
        \operatorname{Cap}_l(Q_d(x)\cap \Eb)
        \leq
        \operatorname{Cap}_l(Q_{\mu,\tau}(x)\cap \Eb) = \operatorname{Cap}_l(\Eb).
    \end{equation*}
    Therefore, for $d\in (0, \delta_{\mu,\tau}(x))$ sufficiently close to $\delta_{\mu,\tau}(x)$,
    \begin{equation*}
        \frac{\mu\left(Q_d(x)\cap \Eb\right)}{\operatorname{Cap}_l(Q_d(x)\cap \Eb)} \geq \frac{\mu\left(Q_d(x)\cap 
        \Eb\right)}{\operatorname{Cap}_l(\Eb)} > \frac{1}{\tau}.
    \end{equation*}
    This contradicts the definition of $\delta_{\mu,\tau}(x)$.
\end{proof}

\section{Eigenvalue counting function. Estimates}
\subsection{Coverings associated with a measure}

In this section, we demonstrate the role of the Otelbaev function in obtaining eigenvalue estimates. We first introduce some 
terminology and recall a useful consequence of the Besicovitch covering lemma, in the form stated in \cite{Guzman}

We say that a covering by cubes $\Xi=\{Q_k\}$ has multiplicity at most $\theta$ if every point of $\mathbb{R}^{\Nb}$ belongs to at 
most 
$\theta$ sets in the covering, in other words,
\begin{equation*}
    \sum_k \chi_{Q_k}(x) \leq \theta
    \quad \text{for all } x \in \mathbb{R}^{\Nb},
\end{equation*}
where $\chi_{Q_k}$ denotes the indicator function of $Q_k$. The covering $\Xi$ is said to have \emph{linkage} at most $\kappa$ if it 
can be split into the union of no more than $\kappa$ families, $\Xi=\bigcup\Xi_j$, each of $\Xi_j$ consisting of disjoint cubes.  Of 
course, the multiplicity of a cover is no greater than its linkage, finite or infinite. The Besicovitch covering lemma establishes 
the following.

\begin{lemma}\label{Besic} Let a cube $Q(x)\in \mathbb{Q}$ be associated with any point $x\in\R^{\Nb}$ and the sizes of cubes $Q(x)$ 
are bounded. Then an at most countable set $\{x_k\}$ can be extracted so that the cubes $Q(x_k)$ form a covering of multiplicity and 
linkage numbers no greater than $\theta$, the latter number depending only on the dimension $\Nb$ but not on the cubes $Q(x)$. 
\end{lemma}
The proof can be found, e.g, in \cite{FrankLaptevWeidl} for a covering of a compact set. The extension for the whole space (or for 
unbounded sets) can be found in \cite{Guzman}.
A direct application of  the Besicovitch lemma to our concrete situation establishes  the following property.
\begin{lemma}\label{l_covering_LO}
Let $\mu$ be a Radon measure on $\mathbb{R}^\Nb$ satisfying Condition \ref{cond.cap}. For fixed numbers $\tau, \eta > 0$  there exist  
constants $\theta$, $\kappa$ depending only on $\Nb$ and a sequence of points $\{x_k\}_{k \in \mathbb{N}}$ in $ \mathbb{R}^\Nb$ such 
that:
\begin{enumerate}
    \item The space $\mathbb{R}^\Nb$ is covered by open cubes
    \begin{equation*}
        Q_k = Q_{d_k}(x_k), \qquad d_k = \min\left\{ \delta_{\theta\mu,\tau}(x_k),\, \eta \right\},
    \end{equation*}
       \item The covering multiplicity of $\{Q_k\}_{k\in\mathbb N}$ does not exceed $\theta$.
    \item The linkage of the covering  does not exceed $\kappa$.
\end{enumerate}
\end{lemma}
\begin{remark}
Clearly, if $\mu$ satisfies Condition~\ref{cond.cap}, then for any fixed
$\theta>0$ the scaled measure $\theta\mu$ also satisfies
Condition~\ref{cond.cap}. In particular, the quantities $\delta_{\theta\mu,\tau}$ and $q^*_{\theta\mu,\tau}$
are well defined for all $\tau$, $\theta>0$. For fixed values of $\theta, \tau,$ the covering may, generally, depend on these 
numbers, however, the multiplicity and the linkage numbers do not depend on them.  
\end{remark} 
Having the Besicovitch covering as above, we associate with it a collection of operators. Namely,  cubes in $\Xi$, 
Proposition~\ref{Prop.form.loc} applied 
to the measure $\theta\mu$, the quadratic forms $\{a_{\theta\mu, Q_k}\}_{Q_k\in \Xi}$ generate self-adjoint semibounded operators 
which we 
denote by $\{H_{\theta\mu}^k\}_{k\in \mathbb{N}}$, suppressing the dependence on $\tau$ and $\eta$. 

\subsection{Eigenvalue estimates}

For a self-adjoint operator $H$, we denote by $N(-\lambda; H)$ the number of eigenvalues of $H$ in $(-\infty,-\lambda)$, counted with 
multiplicity; if the spectrum of $H$ below $-\lambda$ is not discrete, we set $N(-\lambda; H)=\infty$.

In the classical Dirichlet-Neumann bracketing, the eigenvalue counting  function in a larger domain $\Omega$ is estimated from above 
via the sum of counting functions for Neumann problem in smaller domains $\Omega_j$ forming a decomposition of $\Omega$. The 
following 
result, proved in \cite{Rozenblum2022} (extending the one in \cite{FrankLaptevWeidl}), shows that such an upper estimate, with a 
controllable loss in the constant, holds also in a weaker sense, namely, for 
the \emph{decomposition} replaced by a \emph{covering} of finite multiplicity. For the reader’s convenience, we repeat the proof in 
our setting.

\begin{lemma}\label{lem.Nsmaller_sumN_LO}
    Let $\mu$ be a Radon measure on
    $\mathbb{R}^\Nb$ satisfying Condition~\ref{cond.cap}.
    Let $H_\mu$ be the operator generated by the quadratic form~\eqref{quad_form_LO},
    and let $\{H_{\theta\mu}^k\}_{k\in\mathbb{N}}$ be the operators introduced above, the Neumann operators in the cubes. Then, for 
    any $\lambda\geq 0$, the following 
    inequality holds 
    \begin{equation*}
        N(-\lambda;H_\mu) \leq \sum_{k\in \mathbb{N}} N(-\lambda;H_{\theta\mu}^k).
    \end{equation*}
\end{lemma}

\begin{proof}
    Fix $\lambda\geq 0$. For $k\in\mathbb{N}$, define 
    \begin{equation*}
        m_k=N\left(0, H^{k}_{\theta\mu}+\lambda\right)=N(-\lambda;H^{k}_{\theta\mu}).
    \end{equation*}
    If $\sum_{k\in\mathbb{N}}m_k=\infty$, there is nothing to prove. If $m_k>0$, let $\{w_{k,j}\}_{j=1}^{m_k}$ be normalized 
    eigenfunctions of $H^{k}_{\theta\mu}$ corresponding to eigenvalues below $-\lambda$. Extend each $w_{k,j}$ by zero to a 
    function on $\mathbb{R}^\Nb$ (still denoted $w_{k,j}$),
    and define linear functionals on $D(a_\mu)$ by
    \begin{equation*}
        \varphi_{k,j}(u)=(u,w_{k,j})_{L^2(\mathbb{R}^\Nb)}=(u,w_{k,j})_{L^2(Q_k)}.
    \end{equation*}
    Set 
    \begin{equation*}
        \Ls:=\bigcap_{k\in\mathbb{N}}\ \bigcap_{j=1}^{m_k} \ker \varphi_{k,j}.
    \end{equation*}
    Then $\operatorname{codim} \Ls\leq \sum_{k\in\mathbb{N}} m_k$. Fix $u\in \Ls$ and $k\in\mathbb{N}$. We denote by $u^{(k)}$ the 
    restriction of $u$ to the cube $Q_k$. Since $u^{(k)}$ is 
    orthogonal in $L^2(Q_k)$ to $\{w_{k,j}\}_{j=1}^{m_k}$, we derive that 
    \begin{equation*}
        a_{\theta\mu, Q_k}[u^{(k)}] + \lambda \|u^{(k)}\|_{L^2(Q_k)}^2 \geq 0.
    \end{equation*}
    Summing over $k$ gives
    \begin{multline*}
        0\leq \sum_{k\in\mathbb{N}} a_{\theta\mu, Q_k}[u^{(k)}] + \lambda \sum_{k\in\mathbb{N}} 
        \|u^{(k)}\|_{L^2(Q_k)}^2\\
        = \sum_{k\in\mathbb{N}}\int_{Q_k}|\nabla^l u|^2\mathrm{d}x
        -\theta\sum_{k\in\mathbb{N}}\int_{Q_k}|u|^2\mathrm{d}\mu
        +\lambda\sum_{k\in\mathbb{N}}\int_{Q_k}|u|^2\mathrm{d}x.
    \end{multline*}
    Since $\{Q_k\}$ is a cover for $\mathbb{R}^\Nb$ with multiplicity  at most $\theta$, we conclude 
    \begin{equation*}
        0 \leq \theta \left(\int_{\mathbb{R}^\Nb}|\nabla^l u|^2\mathrm{d}x - \int_{\mathbb{R}^\Nb}|u|^2\mathrm{d}\mu + \lambda  
        \int_{\mathbb{R}^\Nb}|u|^2\mathrm{d}x\right) = \theta a_{\mu}[u] + \theta\lambda \|u\|_{L^2(\mathbb{R}^\Nb)}^2.
    \end{equation*}
    Therefore, using the fact that $\operatorname{codim} \Ls\leq \sum_{k\in\mathbb{N}} m_k$ and the Glazman variational principle, we 
    obtain 
    \begin{equation*}
        N(-\lambda;H_\mu) = N(0;H_\mu+\lambda)\leq \sum_{k\in\mathbb{N}} m_k
        =\sum_{k\in\mathbb{N}} N(0;H^{k}_{\theta\mu}+\lambda)
        =\sum_{k\in\mathbb{N}} N(-\lambda;H^{k}_{\theta\mu}),
    \end{equation*}
    as claimed.
\end{proof}

Lemma~\ref{lem.Nsmaller_sumN_LO} will serve to estimate the counting function from above. For the bound from below, we need, by the 
Glazman variational lemma, to find subspaces of functions on which the quadratic form $a_\mu$ is strongly negative. The following 
lemma constructs such a trial function,
localized near any set on which the measure is large relative to its capacity; the construction is based on the capacitary potential 
of Proposition~\ref{prop:cappot}.
\begin{lemma}\label{lem:deep_trial}
Let $\mu$ be a Radon measure on
    $\mathbb{R}^\Nb$ satisfying Condition~\ref{cond.cap}. There exist constants $\tau_0=\tau_0(\Nb,l)>0$ and $c_*=c_*(\Nb,l)>0$ with 
    the following property. Let $\tau\le\tau_0$,
$y\in\R^\Nb$, $\delta>0$, and let $\Eb\subset Q_\delta(y)$ be a compact set with
\begin{equation}\label{eq:Z_large}
   Z_\mu(\Eb) \ge 1/\tau .
\end{equation}
Then there exists a non-zero function $f\in\htt^l(\R^\Nb)$ such that 
$\operatorname{supp} f\subset Q_{8\delta}(y)$ and
\begin{equation}\label{eq:deep}
   a_\mu[f]\ \le\ -\,c_*\,\delta^{-2l}\,\|f\|_{L^2(\R^\Nb)}^2 .
\end{equation}
\end{lemma}

\begin{proof}
Fix a cut-off $\eta\in C_0^\infty(Q_{8\delta}(y))$ with
$0\le\eta\le1$, $\eta=1$ on $Q_{4\delta}(y)$ and
\begin{equation*}
    |D^\gamma\eta|\le C\,\delta^{-|\gamma|}
\end{equation*}
for all $|\gamma|\le l$, where
$C=C(\Nb,l)$. Note that
\begin{equation}\label{eq:dist_annulus}
   \operatorname{dist}\bigl(\operatorname{supp}\nabla\eta,\ \Eb\bigr)
   \ \ge\ 2\delta-\tfrac{\delta}{2}\ >\ \delta,
\end{equation}
since $\operatorname{supp}\nabla\eta\subset
\overline{Q_{8\delta}(y)}\setminus Q_{4\delta}(y)$ and
$\Eb\subset Q_\delta(y)$.

Observe that \eqref{eq:Z_large} implies
$\operatorname{Cap}_l(\Eb)>0$, so Proposition~\ref{prop:cappot} applies. Let
$u_\Eb$ and $\nu_\Eb$ be as in its statement and proof, and set $f:=\eta\,u_\Eb $. Since $u_\Eb$ is bounded by Proposition 
\ref{prop:cappot} and $\nabla^l u_\Eb\in L^2(\R^\Nb)$,
we conclude that $u_\Eb\in \Ht^l_{\mathrm{loc}}(\R^\Nb)$. Hence, $f\in \Ht^l(\R^\Nb)$ is compactly supported, and therefore
$f\in\htt^l(\R^\Nb)$, with $\operatorname{supp}f\subset Q_{8\delta}(y)$.

We show that \eqref{eq:deep} holds. First we estimate the potential term. Since 
\begin{equation*}
    \Eb\subset Q_\delta(y)\subset Q_{4\delta}(y),
\end{equation*}
we have $f=u_\Eb$
on $\Eb$, and by Proposition \ref{prop:cappot}, $f\ge1$ on $\Eb\setminus e$, where
$e:=\Eb\cap\{u_\Eb<1\}$ is a Borel set with
$\operatorname{Cap}_l(e)=0$. Hence Lemma~\ref{lem:mu_ll_cap} gives $\mu(e)=0$. Hence, by
\eqref{eq:Z_large},
\begin{equation}\label{eq:pot_lower}
   \int_{\R^\Nb}|f|^2\mathrm{d}\mu \ge \mu(\Eb\setminus e) = \mu(\Eb)
    \ge \frac{1}{\tau} \operatorname{Cap}_l(\Eb).
\end{equation}

Next, we estimate the kinetic term. By the Leibniz rule, for $|\alpha|=l$,
\[
D^\alpha f=\eta\,D^\alpha u_\Eb
+\sum_{0\le\beta<\alpha}\binom{\alpha}{\beta}\,D^{\alpha-\beta}\eta\ D^\beta u_\Eb .
\]
The first term is estimated by \eqref{cp:energy} in Proposition \ref{prop:cappot}:
\begin{equation*}
    \|\eta\nabla^l u_\Eb\|_{L^2}^2\le C_0\operatorname{Cap}_l(\Eb).
\end{equation*}
Every remaining term is supported in $\operatorname{supp}\nabla\eta$, where, by
\eqref{eq:dist_annulus} and \eqref{eq:cp_decay} with
$j=|\beta|$,
\[
\bigl|D^{\alpha-\beta}\eta\ D^\beta u_\Eb\bigr|
 \lesssim \delta^{-(l-|\beta|)}\cdot
\operatorname{Cap}_l(\Eb) \delta^{2l-\Nb-|\beta|}
 = \operatorname{Cap}_l(\Eb) \delta^{l-\Nb}.
\]
Integrating the square over $Q_{8\delta}(y)$ and using the monotonicity and scaling of the capacity,
\[
\operatorname{Cap}_l(\Eb)\le
\operatorname{Cap}_l\bigl(\overline{Q_\delta(y)}\bigr)\simeq\delta^{\Nb-2l},
\]
we obtain
\[
\bigl\|D^{\alpha-\beta}\eta D^\beta u_\Eb\bigr\|_{L^2}^2
 \lesssim \operatorname{Cap}_l(\Eb)^2\delta^{2l-\Nb}
 \lesssim \operatorname{Cap}_l(\Eb).
\]
Altogether,
\begin{equation}\label{eq:energy_upper}
   \|\nabla^l f\|_{L^2(\R^\Nb)}^2 \le C_1\operatorname{Cap}_l(\Eb),
   \qquad C_1=C_1(\Nb,l).
\end{equation}

Next, we estimate the $L^2$ norm. By \eqref{cp:bdd} in Proposition \ref{prop:cappot} and \eqref{eq:u_repr},
\[
\|f\|_{L^2}^2
\le \int_{Q_{8\delta}(y)}u_\Eb^2 \mathrm{d}x
\lesssim \int_{Q_{8\delta}(y)}u_\Eb \mathrm{d}x
 \lesssim \int_\Eb\Bigl(\int_{Q_{8\delta}(y)}|x-z|^{2l-\Nb}\mathrm{d}x\Bigr)
\mathrm{d}\nu_\Eb(z).
\]
For $z\in\Eb\subset Q_\delta(y)$ we have
$Q_{8\delta}(y)\subset B_{C\delta}(z)$, and since $2l>0$,
\[
\int_{Q_{8\delta}(y)}|x-z|^{2l-\Nb} \mathrm{d}x
\le\int_{B_{C\delta}(z)}|x-z|^{2l-\Nb} \mathrm{d}x
\ \lesssim\ \delta^{2l}.
\]
Together with $\nu_\Eb(\R^\Nb)\simeq\operatorname{Cap}_l(\Eb)$
(see \eqref{eq:thm227} and Remark~\ref{rem:S_R_capacities}), this yields
\begin{equation}\label{eq:L2_upper}
   \|f\|_{L^2(\R^\Nb)}^2 \le C_2\delta^{2l} \operatorname{Cap}_l(\Eb),
   \qquad C_2=C_2(\Nb,l).
\end{equation}

Set $\tau_0:=(2C_1)^{-1}$. For $\tau\le\tau_0$, combining
\eqref{eq:pot_lower} and \eqref{eq:energy_upper},
\[
a_\mu[f]
 \le \Bigl(C_1-\frac1\tau\Bigr)\operatorname{Cap}_l(\Eb)
 \le -\frac{1}{2\tau}\,\operatorname{Cap}_l(\Eb)
 \le -C_1\operatorname{Cap}_l(\Eb),
\]
and then \eqref{eq:L2_upper} gives
\[
a_\mu[f] \le -\frac{C_1}{C_2}\delta^{-2l}\|f\|_{L^2(\R^\Nb)}^2,
\]
which is \eqref{eq:deep} with $c_*:=C_1/C_2$.

Finally, since $\operatorname{Cap}_l(\Eb)>0$, the set
$\Eb\setminus e$ is nonempty, so $u_\Eb(x_0)\ge1$ at some
$x_0\in\Eb$; by lower semicontinuity $u_\Eb>\tfrac12$ on a neighbourhood of
$x_0$, and hence, $\|f\|_{L^2}>0$.
\end{proof}

Next, we introduce the following notation for the superlevel sets:
\begin{equation*}
    M_{\mu, \tau}(t):=\{x\in\R^\Nb: q_{\mu,\tau}^*(x)\ge t\},\qquad t>0 .
\end{equation*}

Now we arrive at our first theorem on eigenvalue estimates.

\begin{theorem}\label{t_est_distr_func_LO}
Let $\mu$ be a Radon measure on $\mathbb{R}^\Nb$ satisfying
Condition~\ref{cond.cap}, and let $H_\mu$ be the operator generated by
the quadratic form~\eqref{quad_form_LO}. Let $\rho,\theta,\varpi,\tau_0$
be the constants from Lemmas~\ref{l_Poincare}, \ref{l_covering_LO},
\ref{Mazya_thm_11.3}, and \ref{lem:deep_trial}, respectively, and set
\[
    \tau_1:=\tau_0/2,\qquad \tau_2:=\theta\varpi(1+\rho).
\]
Then there exist constants $C_1,C_2,A_1,A_2>0$, depending only on
$\Nb$ and $l$, such that, for every $\lambda>0$,
\begin{equation}\label{weak_CLR}
    C_1 \lambda^{\frac{\Nb}{2l}}
    \mathcal{L}\bigl(M_{\mu,\tau_1}(A_1\lambda)\bigr)
    \leq
    N(-\lambda;H_\mu)
    \leq
    C_2
    \int_{M_{\mu,\tau_2}(A_2\lambda)}
    \bigl(q_{\mu,\tau_2}^*(y)\bigr)^{\frac{\Nb}{2l}}\,\mathrm{d}y.
\end{equation}
The lower and upper bounds in \eqref{weak_CLR} are understood in the extended
sense and may be infinite for some values of $\lambda>0$.
\end{theorem}

The upper estimate in \eqref{weak_CLR} can be considered as a version of the CLR estimate, adapted to general  measure-potentials. In 
fact, the inequality 
\eqref{weak_CLR} can be understood as an estimate of the number of eigenvalues of the Schr{\"o}dinger operator below $-\lambda$ via 
the volume of the region in the phase space where the \emph{effective Hamiltonian} $|\xi|^{2l}-q_{\mu,\theta\varpi(1+\rho)}^*(x)$, 
with the Otelbaev function acting as a potential, is less than $-A_2\lambda$.

\begin{proof}
    \textbf{The upper bound.} Since  $\lambda>0$, we can set
    \begin{equation*}
        \tau = \varpi(1+\rho), 
        \qquad
        \eta^{2l} > \frac{1}{(1+\rho)\lambda}.
    \end{equation*}
    Let $\{Q_k\}_{k\in\mathbb{N}}$ be the cubes from Lemma \ref{l_covering_LO}. We denote their centers and edgelengths by 
    $\{x_k\}_{k\in\mathbb{N}}$ and $\{d_k\}_{k\in\mathbb{N}}$, respectively. By \eqref{Lem_2.5} we obtain
    \begin{equation*}
        a_{\theta\mu,Q_k} [f]  \geq \left( 1 - \varpi \sup_{\Eb\subset Q_k}  Z_{\theta \mu}(\Eb)  \right) \|\nabla^l f\|_{L^2(Q_k)}^2
        - \varpi \sup_{\Eb\subset Q_k}  Z_{\theta \mu}(\Eb) d_k^{-2l}\| f\|_{L^2(Q_k)}^2.
    \end{equation*}
    Since $d_k = \min\left\{ \delta_{\theta\mu,\tau}(x_k),\, \eta \right\}$, it follows from Lemma \ref{est_for_small_cubes} that
    \begin{equation*}
        \sup_{\Eb\subset Q_k}  Z_{\theta \mu}(\Eb) \leq \sup_{\Eb\subset Q_{\theta\mu,\tau}(x_k)}  Z_{\theta \mu}(\Eb) 
        \leq \frac{1}{\tau}.
    \end{equation*}
    Therefore,
    \begin{align*}
        a_{\theta\mu,Q_k} [f]  &\geq \left( 1 - \frac{\varpi}{\tau}\right) \|\nabla^l f\|_{L^2(Q_k)}^2 
        - \frac{\varpi}{\tau} d_k^{-2l}\| f\|_{L^2(Q_k)}^2\\
        & \geq \left( 1 - \frac{1}{1+\rho}\right) \|\nabla^l f\|_{L^2(Q_k)}^2 
        - \frac{1}{(1+\rho)d_k^{2l}} \| f\|_{L^2(Q_k)}^2.
    \end{align*}
    In particular, we conclude that 
    \begin{equation*}
        N(-\lambda; H_{\theta\mu}^k) = 0,
        \qquad
        \text{if } -\lambda < - \frac{1}{(1+\rho)d_k^{2l}}.
    \end{equation*}
    Moreover, by Lemma \ref{l_Poincare}, for $f\in \mathbb{P}_{l-1}(\mathbb{R}^\Nb)^{\perp}$  in $L^2(Q_k)$ sense, we obtain
    \begin{equation*}
        a_{\theta\mu,Q_k} [f]  \geq \left( 1 - \frac{1}{1+\rho}\right) \frac{1}{\rho d_k^{2l}} \| f\|_{L^2(Q_k)}^2 - 
        \frac{1}{(1+\rho)d_k^{2l}} \| f\|_{L^2(Q_k)}^2 = 0.
    \end{equation*}
    Therefore, 
    \begin{equation*}
        N(-\lambda; H_{\theta\mu}^k) \leq P_{\Nb,l} := \dim \mathbb{P}_{l-1}(\mathbb{R}^\Nb).
    \end{equation*}

    Set
    \begin{equation*}
        \mathcal J_\lambda
        :=
        \left\{
        k\in\mathbb N:
        -\lambda \geq - \frac{1}{(1+\rho)d_k^{2l}}
        \right\}.
    \end{equation*}
    Hence, by Lemma \ref{lem.Nsmaller_sumN_LO}, we conclude 
    \begin{equation}\label{ineq: Nleq_sumbeta0}
        N(-\lambda; H_\mu) \leq \sum_{k\in \mathcal J_\lambda}P_{\Nb,l}.
    \end{equation}
    Thus the right-hand side of \eqref{ineq: Nleq_sumbeta0} is
    $P_{\Nb,l}|\mathcal J_\lambda|$, where $|\mathcal J_\lambda|\in\mathbb N\cup\{\infty\}$ is the number of cubes contributing to 
    the sum. The inequality is understood in the extended sense. We do not assume in advance that $\mathcal J_\lambda$ is finite. If 
    $\mathcal J_\lambda$ is infinite, the estimates below show that the integral on the right-hand side of the desired estimate is 
    infinite, and hence the desired estimate is trivial. Therefore, in the non-trivial case, we may assume that $\mathcal J_\lambda$ 
    is finite.

    Since $|Q_k|=d_k^{\Nb}$, \eqref{ineq: Nleq_sumbeta0} can be rewritten as
    \begin{equation}\label{ineq: Nleq_sumbeta}
        N(-\lambda; H_\mu) \leq   \sum_{k\in \mathcal J_\lambda} P_{\Nb,l} \int_{Q_k} \frac{1}{d_k^{\Nb}}\mathrm{d}y.
    \end{equation}

    For $k\in \mathcal J_\lambda$, we have
    \begin{equation*}
        d_k^{2l} \leq  \frac{1}{(1+\rho) \lambda} < \eta^{2l}.
    \end{equation*}
    Since $d_k$ is defined as $\min\left\{ \delta_{\theta\mu,\tau}(x_k), \eta \right\}$, it follows that
    \begin{equation}\label{eq:d_k_equal_dthetamu}
        d_k = \delta_{\theta\mu,\tau}(x_k),
        \qquad
        k\in \mathcal J_\lambda.
    \end{equation}
    Therefore, \eqref{ineq: Nleq_sumbeta} implies that
    \begin{equation*}
        N(-\lambda; H_\mu) \leq  \sum_{k\in \mathcal J_\lambda} P_{\Nb,l} \int_{Q_k} 
        \frac{1}{\delta_{\theta\mu,\tau}^{\Nb}(x_k)}\mathrm{d}y
        \leq P_{\Nb,l} \sum_{k\in \mathcal J_\lambda} \int_{Q_k} 
        (q_{\theta\mu,\tau}^*(x_k))^{\frac{\Nb}{2l}} \mathrm{d}y.
    \end{equation*}
    Moreover, by \eqref{eq:d_k_equal_dthetamu} and Lemma \ref{l_cont_var_LO}, for any $k\in \mathcal{J}_\lambda$ and any $y\in Q_k$,
    \begin{equation}\label{eq:y_vs_xk}
        q_{\theta\mu,\tau}^*(y) \geq \frac{q_{\theta\mu,\tau}^*(x_k)}{2^{2l}}  \geq \frac{1}{2^{2l}} 
        \frac{1}{d_k^{2l}} \geq \frac{1 + \rho}{2^{2l}}  \lambda.
    \end{equation}
    Therefore, we obtain
    \begin{equation*}
        N(-\lambda; H_\mu) \leq 2^{\Nb} P_{\Nb,l} \sum_{k\in \mathcal J_\lambda} \int_{Q_k} 
        (q_{\theta\mu,\tau}^*(y))^{\frac{\Nb}{2l}} \mathrm{d}y.
    \end{equation*}
    Moreover, \eqref{eq:y_vs_xk} implies that
    \begin{equation*}
        \bigcup_{k\in \mathcal{J}_\lambda} Q_k \subset \left\{x\in \mathbb{R}^\Nb: q_{\theta\mu,\tau}^*(x) \geq \frac{1 + 
        \rho}{2^{2l}}  
        \lambda \right\}.
    \end{equation*}
    Therefore, taking into account the cover multiplicity bound $\theta$, we obtain
    \begin{equation*}
         N(-\lambda; H_\mu) \leq 2^{\Nb} P_{\Nb,l} \theta \int_{\left\{x\in \mathbb{R}^\Nb: q_{\theta\mu,\tau}^*(x) \geq \frac{1 + 
         \rho}{2^{2l}}  
        \lambda \right\}} (q_{\theta\mu,\tau}^*(y))^{\frac{\Nb}{2l}} \mathrm{d}y.
    \end{equation*}
    Finally, recalling that $q^*_{\theta\mu,\tau} = q^*_{\mu, \theta \tau}$, we conclude the proof of the upper bound.

    \textbf{The lower bound.} Fix $\lambda>0$ and let $c_*$ be the constant from Lemma \ref{lem:deep_trial}. Set $A_1 := 
    2^{2l+1}/c_*$ and $d: = (A_1\lambda)^{-\frac{1}{2l}}$. We start by constructing, for  any point $x\in 
    M_{\mu,\tau_1}(A_1\lambda)$, a trial function $f_x\in \htt^l(\R^\Nb)$ supported in the cube $16Q_{\mu,\tau_1}(x)$, satisfying
      \begin{equation}\label{eq:below}
   a_\mu[f_x] \le -2\lambda\|f_x\|_{L^2(\mathbb{R}^\Nb)}^2 < -\lambda\|f_x\|_{L^2(\mathbb{R}^\Nb)}^2.
\end{equation}
       By Condition~\ref{cond.cap}, for some  $d_0>0$,
\begin{equation*}
    \sup_{x\in\R^\Nb}\sup_{\Eb\subset Q_{d_0}(x)}Z_\mu(\Eb)<\frac{1}{\tau_1},
\end{equation*}
and therefore, by the definition of $\delta_{\mu,\tau_1}(\cdot)$, we have
$\delta_{\mu,\tau_1}(x) \ge d_0 > 0$ for all $x\in\R^\Nb$. On the other hand, for
$x\in M_{\mu,\tau_1}(A_1\lambda)$, we have $\delta_{\mu,\tau_1}(x)\le d$. Thus,
\begin{equation}\label{eq:delta_small}
   0<d_0 < \delta_{\mu,\tau_1}(x) \le d .
\end{equation}
 Further, the definition of $\delta_{\mu,\tau_1}(x)$ implies
\begin{equation*}
    \sup_{\substack{\Eb\subset 2Q_{\mu,\tau_1}(x)}}
   Z_\mu(\Eb) \ge \frac{1}{\tau_1} = \frac{2}{\tau_0}
    > \frac{3}{2\tau_0}.
\end{equation*}
Hence there exists a Borel set $\Eb\subset 2Q_{\mu,\tau_1}(x)$ with
$Z_\mu(\Eb)>\frac{3}{2\tau_0}$. Note that
$\mu(\Eb)=Z_\mu(\Eb)\operatorname{Cap}_l(\Eb)>0$ and $\mu(\Eb)<\infty$, since $\Eb$ is
bounded. By the inner regularity property  of $\mu$, we can choose a compact set $\Eb_x\subset \Eb$
with $\mu(\Eb_x)\ge\frac23\mu(\Eb)>0$. Then, by Lemma \ref{lem:mu_ll_cap} and monotonicity of the capacity,
\begin{equation*}
    Z_\mu(\Eb_x) \ge \frac{2}{3} Z_\mu(\Eb) > \frac{1}{\tau_0}.
\end{equation*}
Now,  by Lemma~\ref{lem:deep_trial}, used for $y=x$, the cube $2\delta_{\mu,\tau_1}(x)$
and $\tau=\tau_0$, we obtain a non-zero function $f_x\in\htt^l(\R^\Nb)$ satisfying
\begin{equation}\label{eq:supp_and_depth}
   \operatorname{supp} f_x\subset 16Q_{\mu,\tau_1}(x),
   \qquad
   a_\mu[f_x] \le -c_*(2\delta_{\mu,\tau_1}(x))^{-2l}\|f_x\|_{L^2}^2 .
\end{equation}
Due to \eqref{eq:delta_small} and the definition of $A_1$, this yields
\eqref{eq:below}.

We now extract from the set $\left\{ f_x\right\}$  a sufficiently large family of trial functions with  disjoint supports
by means of a covering argument. Consider the cubes
\begin{equation*}
    Q(x) := 16\,Q_{\mu,\tau_1}(x),
    \qquad x\in M_{\mu,\tau_1}(A_1\lambda).
\end{equation*}
Then, the family $\{Q(x)\}_{x\in M_{\mu,\tau_1}(A_1\lambda)}$ is a
covering of $M_{\mu,\tau_1}(A_1\lambda)$.

By the 'linkage' part of the  Besicovitch covering theorem, there exist a constant $K=K(\Nb)\in\mathbb N$ and  at most countable 
family of cubes
\begin{equation*}
    \Xi=\{Q(x_j)\}_{j\in\mathcal I},
    \qquad x_j\in M_{\mu,\tau_1}(A_1\lambda),
\end{equation*}
that still covers $M_{\mu,\tau_1}(A_1\lambda)$ and decomposes into $K$ subfamilies
\begin{equation*}
    \Xi=\bigcup_{k=1}^{K}\Xi_k,
\end{equation*}
where, for each $k$, the cubes in $\Xi_k$ are pairwise disjoint.

Since $\Xi$ covers $M_{\mu,\tau_1}(A_1\lambda)$ and, by \eqref{eq:delta_small},
$\mathcal L(Q)\le(16d)^\Nb$ for every $Q\in\Xi$, we obtain
\begin{equation*}
    \mathcal L\bigl(M_{\mu,\tau_1}(A_1\lambda)\bigr)
    \le \sum_{Q\in\Xi}\mathcal L(Q)
    \le |\Xi| (16d)^\Nb.
\end{equation*}
Recalling
that $d=(A_1\lambda)^{-\frac{1}{2l}}$, this yields
\begin{equation}\label{eq:Xi_lower}
    |\Xi|
    \ge 16^{-\Nb}\,d^{-\Nb}\,\mathcal L\bigl(M_{\mu,\tau_1}(A_1\lambda)\bigr)
    = 16^{-\Nb}\,A_1^{\frac{\Nb}{2l}}\,\lambda^{\frac{\Nb}{2l}}\,
    \mathcal L\bigl(M_{\mu,\tau_1}(A_1\lambda)\bigr).
\end{equation}

Since $\Xi$ is a union of finitely many
subfamilies, there is an index $k_0\in\{1,\dots,K\}$ such that
$|\Xi_{k_0}|=\max_{1\le k\le K}|\Xi_k|$, and for this index
\begin{equation}\label{eq:pigeon_besi}
    |\Xi_{k_0}|\ge \frac{1}{K}\,|\Xi|.
\end{equation}

Set
\begin{equation*}
    \Ls:=\operatorname{span}\{f_{x_j}:\ Q(x_j)\in\Xi_{k_0}\}.
\end{equation*}
The cubes in $\Xi_{k_0}$ are pairwise disjoint, so by \eqref{eq:supp_and_depth}
the functions $\{f_{x_j}:Q(x_j)\in\Xi_{k_0}\}$ have pairwise disjoint
supports; in particular, they are linearly independent, and
$\dim\Ls=|\Xi_{k_0}|$. Moreover, for every
$f=\sum_{j}a_j f_{x_j}\in\Ls\setminus\{0\}$, the disjointness of the supports
together with \eqref{eq:below} gives
\begin{equation*}
    a_\mu[f]=\sum_j|a_j|^2 a_\mu[f_{x_j}]
    \le -2\lambda\sum_j|a_j|^2\|f_{x_j}\|_{L^2(\mathbb{R}^\Nb)}^2
    = -2\lambda\|f\|_{L^2(\mathbb{R}^\Nb)}^2
    < -\lambda\|f\|_{L^2(\mathbb{R}^\Nb)}^2 .
\end{equation*}
Hence, the Glazman variational lemma implies
\begin{equation*}
    |\Xi_{k_0}|=\dim\Ls\le N(-\lambda;H_\mu).
\end{equation*}
Combining this with \eqref{eq:pigeon_besi} and \eqref{eq:Xi_lower}, we complete the proof.
\end{proof}

\begin{remark}
    As mentioned earlier, the Otelbaev function does not produce estimates for $N(0;H_\mu)$. In fact, if $q^*_{\mu,\tau_2}\not\equiv 
    0$, then
$q^*_{\mu,\tau_2}$ cannot decay faster than
$|x|^{-2l}$ at infinity, and the phase space volume is infinite. Consequently, the estimate is contentless for estimating 
$N(0;H_\mu)$.
\end{remark}

An immediate consequence of the upper bound obtained in Theorem~\ref{t_est_distr_func_LO}
is the condition of  discreteness of the negative spectrum of $H_\mu$.
Since the eigenvalue counting function is controlled by the distribution function of
the Otelbaev function $q_{\mu,\tau}^*$, it is natural to describe when the
corresponding superlevel-set integrals are finite for $\lambda>0$.
This is done in the following lemma.

\begin{lemma}\label{lem: prep_discr_fixed}
Let $\mu$ be a Radon measure satisfying Condition~\ref{cond.cap}, and let
$\tau>0$. For $\lambda>0$, set
\begin{equation*}
    I(\lambda)
    =
    \int_{M_{\mu,\tau}(\lambda)}
    \left(q_{\mu,\tau}^*(x)\right)^{\frac{\Nb}{2l}} \mathrm{d}x .
\end{equation*}
If
\begin{equation*}
        \limsup_{|x|\to\infty} q^*_{\mu,\tau}(x)<\lambda,
\end{equation*}
then $I(\lambda)<\infty$. Conversely, if
\begin{equation*}
        \limsup_{|x|\to\infty} q^*_{\mu,\tau}(x)>2^{2l}\lambda,
\end{equation*}
then $I(\lambda)=\infty$.
\end{lemma}

\begin{remark}
The gap between the constants in the above necessary and sufficient conditions appears to be unavoidable. This phenomenon is already 
present for regular
potentials in the early work of Maz'ya~\cite{Maz1964}; see also \cite[Sections~18.1 and~18.2]{Mazya}.
\end{remark}

\begin{proof}
Assume first that
\begin{equation*}
    \limsup_{|x|\to\infty} q^*_{\mu,\tau}(x)<\lambda .
\end{equation*}
Then there exist $R>0$ and $\varepsilon>0$ such that
\begin{equation*}
    q^*_{\mu,\tau}(x)\leq \lambda-\varepsilon,
    \qquad |x|>R.
\end{equation*}
Therefore,
\begin{equation*}
    M_{\mu,\tau}(\lambda)
    \subset
    \{x\in\mathbb R^\Nb:\ |x|\leq R\}.
\end{equation*}
By Corollary~\ref{cor:continuity_qstar}, the function
$q^*_{\mu,\tau}$ is continuous. Hence it is bounded on the compact set
$\{|x|\leq R\}$, and consequently $I(\lambda)<\infty$.

Conversely, assume that
\begin{equation*}
    \limsup_{|x|\to\infty} q^*_{\mu,\tau}(x)>2^{2l}\lambda .
\end{equation*}
Then there exist $\varepsilon>0$ and a sequence $\{x_k\}_{k\in\mathbb N}$ with
$|x_k|\to\infty$ such that
\begin{equation*}
    q^*_{\mu,\tau}(x_k)>2^{2l}\lambda+\varepsilon .
\end{equation*}
Equivalently,
\begin{equation*}
    \delta_{\mu,\tau}(x_k)
    <
    \left(2^{2l}\lambda+\varepsilon\right)^{-\frac{1}{2l}}.
\end{equation*}
Passing to a subsequence if necessary, we may assume that the cubes $Q_k:=Q_{\mu,\tau}(x_k)$ are pairwise disjoint.

By Lemma~\ref{l_cont_var_LO}, for every $y\in Q_k$ we have
\begin{equation*}
    q^*_{\mu,\tau}(y)
    \geq
    2^{-2l}q^*_{\mu,\tau}(x_k)
    >
    \lambda .
\end{equation*}
Thus $Q_k\subset M_{\mu,\tau}(\lambda)$. Moreover,
\begin{multline*}
    \int_{Q_k}
    \left(q^*_{\mu,\tau}(y)\right)^{\frac{\Nb}{2l}} \mathrm{d}y
    \geq
    \int_{Q_k}
    \left(2^{-2l}q^*_{\mu,\tau}(x_k)\right)^{\frac{\Nb}{2l}} \mathrm{d}y  
    =
    2^{-\Nb}
    \left(q^*_{\mu,\tau}(x_k)\right)^{\frac{\Nb}{2l}}
    |Q_k|  \\
    =
    2^{-\Nb}
    \delta_{\mu,\tau}(x_k)^{-\Nb}
    \delta_{\mu,\tau}(x_k)^\Nb  
    =
    2^{-\Nb}.
\end{multline*}
Since the cubes $Q_k$ are pairwise disjoint and contained in the superlevel
set $\{q^*_{\mu,\tau}\geq\lambda\}$, we obtain
\begin{equation*}
    I(\lambda)
    \geq
    \sum_{k=1}^{\infty}
    \int_{Q_k}
    \left(q^*_{\mu,\tau}(y)\right)^{\frac{\Nb}{2l}} \mathrm{d}y
    \geq
    \sum_{k=1}^{\infty}2^{-\Nb}
    =
    \infty .
\end{equation*}
This proves the converse implication and completes the proof.
\end{proof}

As an immediate consequence, we recover the following criterion for the
finiteness of all superlevel-set integrals.

\begin{corollary}\label{cor: prep_discr}
     Let $\mu$ be a Radon measure on $\mathbb{R}^\Nb$ satisfying Condition~\ref{cond.cap}, and let $\tau>0$. Then,
    \begin{equation*}
        \int_{M_{\mu,\tau}(\lambda)} (q_{\mu,\tau}^*(y))^{\frac{\Nb}{2l}} \mathrm{d}y < \infty
    \end{equation*}
    for all $\lambda>0$ if and only if 
    \begin{equation*}
        \lim_{|x|\to \infty}q_{\mu,\tau}^*(x)=0.
    \end{equation*}
\end{corollary}

\begin{proof}
    Follows from Lemma \ref{lem: prep_discr_fixed}.
\end{proof}

Combining the eigenvalue estimate from
Theorem~\ref{t_est_distr_func_LO} with Corollary~\ref{cor: prep_discr},
we obtain the following sufficient and necessary conditions for discreteness of the
negative spectrum.
\begin{theorem}\label{dicretness}
    Let $\mu$ be a Radon measure on $\mathbb{R}^\Nb$ satisfying
    Condition~\ref{cond.cap}, let $H_\mu$ be the operator generated by the
    quadratic form~\eqref{quad_form_LO}, and let $\tau_1$, $\tau_2$ be the
    constants from Theorem~\ref{t_est_distr_func_LO}.
    \begin{enumerate}
        \item\label{discr:suf} If
        $\displaystyle\lim_{|x|\to \infty}q_{\mu,\tau_2}^*(x)=0$, then the
        negative spectrum of $H_\mu$ is discrete.
        \item\label{discr:nec} Conversely, if the negative spectrum of
        $H_\mu$ is discrete, then
        $\displaystyle\lim_{|x|\to \infty}q_{\mu,\tau_1}^*(x)=0$.
    \end{enumerate}
\end{theorem}

\begin{proof}
    Statement \eqref{discr:suf} follows from the upper bound in
    Theorem~\ref{t_est_distr_func_LO} and Corollary~\ref{cor: prep_discr}.

    We prove \eqref{discr:nec}. Note first that the negative spectrum of
    $H_\mu$ is discrete if and only if $N(-\lambda;H_\mu)<\infty$ for every
    $\lambda>0$. Assume that $q_{\mu,\tau_1}^*$ does not tend to zero at
    infinity. Then there exist $\varepsilon>0$ and a sequence
    $\{x_k\}_{k\in\mathbb N}$ with $|x_k|\to\infty$ such that $q_{\mu,\tau_1}^*(x_k)\geq\varepsilon$ for $k\in\mathbb N$; in 
    particular, $\delta_{\mu,\tau_1}(x_k)<\infty$. By
    Condition~\ref{cond.cap}, there is $d_0>0$ such that
    $\delta_{\mu,\tau_1}(x)\geq d_0$ for all $x\in\mathbb R^\Nb$. Passing to
    a subsequence, we may assume that the cubes
    $\{Q_{d_0}(x_k)\}_{k\in\mathbb N}$ are pairwise disjoint. By
    Lemma~\ref{l_cont_var_LO}, for every
    $y\in Q_{\mu,\tau_1}(x_k)\supset Q_{d_0}(x_k)$,
    \begin{equation*}
        q_{\mu,\tau_1}^*(y)
        \geq
        2^{-2l}q_{\mu,\tau_1}^*(x_k)
        \geq
        2^{-2l}\varepsilon .
    \end{equation*}
    Hence $Q_{d_0}(x_k)\subset M_{\mu,\tau_1}(2^{-2l}\varepsilon)$ for every
    $k$, and by disjointness
    \begin{equation*}
        \mathcal L\bigl(M_{\mu,\tau_1}(2^{-2l}\varepsilon)\bigr)
        \geq
        \sum_{k\in\mathbb N} d_0^{\,\Nb}
        =
        \infty .
    \end{equation*}
    Applying the lower bound of Theorem~\ref{t_est_distr_func_LO} with
    $\lambda=2^{-2l}\varepsilon/A_1>0$, we obtain
    $N(-\lambda;H_\mu)=\infty$, so the negative spectrum of $H_\mu$ is not
    discrete. This proves \eqref{discr:nec}.
\end{proof}

\begin{remark}
Theorem~\ref{dicretness} implies, in particular, that if $\mathcal{L}(M_{\mu,\tau}(\lambda))<\infty$ for every $\lambda>0$, then the 
negative spectrum of $H_\mu$ is discrete. This can be viewed as an
analogue of \cite[Theorem~1]{Simon2008} for negative singular potentials.
\end{remark}

\begin{remark}
Corollary~\ref{cor: prep_discr} shows that the finiteness of the phase-space volume integrals appearing in 
Theorem~\ref{t_est_distr_func_LO} for all $\lambda>0$ is equivalent to the vanishing of $q_{\mu,\tau}^*$ at infinity.
Therefore, within the approach based on Theorem~\ref{t_est_distr_func_LO},
the condition in Theorem~\ref{dicretness} is essentially optimal as a
sufficient condition for discreteness.
\end{remark}

\section{Lieb-Thirring-type estimate}

\subsection{General estimates}
We are now in a position to establish a Lieb-Thirring-type estimate for operators $H_\mu$ with Radon measure-potentials satisfying 
Condition~\ref{cond.cap}. The estimate is formulated in terms of the associated Otelbaev function. Throughout this section, we assume 
that the negative spectrum of $H_\mu$ is discrete. This is guaranteed, in particular, provided that the conditions of 
Theorem~\ref{dicretness} are satisfied.

\begin{theorem}\label{thm: LT_for_LO}
    Let $\mu$ be a Radon measure on $\mathbb{R}^\Nb$ satisfying Condition~\ref{cond.cap}, and let $H_\mu$ be the operator generated 
    by the quadratic form~\eqref{quad_form_LO}. Let $\tau_1$ and $\tau_2$ be the constants from Theorem \ref{t_est_distr_func_LO}. 
    Let $\gamma>0$ and suppose that $q_{\mu,\tau_2}^* \in
        L^{\frac{\Nb}{2l}+\gamma}(\mathbb{R}^\Nb)$. Then the negative spectrum of $H_\mu$ is discrete. Moreover, if
    $\{\lambda_k\}$ denotes the negative eigenvalues of $H_\mu$, counted with multiplicity, then
    \begin{equation}\label{Otelbaev.LT}
         b_1 \int_{\R^\Nb}
   (q^*_{\mu,\tau_1}(x))^{\frac{\Nb}{2l}+\gamma} \mathrm{d}x \leq \sum_k |\lambda_k|^\gamma \leq b_2 \int_{\mathbb{R}^\Nb} (q_{\mu, 
   \tau_2}^*(y))^{\frac{\Nb}{2l} + \gamma} \mathrm{d}y,
    \end{equation}
    for some constants $b_1$, $b_2>0$ depending only on $\Nb$, $l$, and $\gamma$.
\end{theorem}

The sum in \eqref{Otelbaev.LT} will further on be referred to as 'the LT sum'.

\begin{remark}\label{rem:LT_lower_unconditional}
The integrability assumption on $q^*_{\mu,\tau_2}$ is used only in the upper
bound and in establishing the discreteness of the negative spectrum. As the
proof shows, the lower bound in \eqref{Otelbaev.LT} holds whenever the
negative spectrum of $H_\mu$ is discrete.
\end{remark}

\begin{proof}
\textbf{The upper bound.} By Theorem \ref{t_est_distr_func_LO}, we estimate 
    \begin{multline*}
        \sum_k |\lambda_k|^\gamma = \gamma\int_{-\infty}^0 |\lambda|^{\gamma - 1} N(\lambda;H_\mu) \mathrm{d}\lambda \\
        \leq \gamma C_2 \int_{-\infty}^0 |\lambda|^{\gamma - 1} \int_{M_{\mu, \tau_2}(A_2|\lambda|)} 
        (q_{\mu,\tau_2}^*(y))^{\frac{\Nb}{2l}} \mathrm{d}y \mathrm{d}\lambda.
    \end{multline*}
    Using Fubini’s theorem, we derive
    \begin{multline*}
        \sum_k |\lambda_k|^\gamma \leq \gamma C_2 \int_{\mathbb{R}^\Nb} (q_{\mu, \tau_2}^*(y))^{\frac{\Nb}{2l}} 
        \int_{- \frac{1}{A_2}q_{\mu,\tau_2}^*(y)}^0 |\lambda|^{\gamma - 1} \mathrm{d}\lambda \mathrm{d}y\\
        \leq C_2 \left( \frac{1}{A_2} \right)^\gamma \int_{\mathbb{R}^\Nb} (q_{\mu,\tau_2}^*(y))^{\frac{\Nb}{2l} + 
        \gamma} \mathrm{d}y.
    \end{multline*}
    This gives the upper bound.

\textbf{The lower bound.} Set $p:=\tfrac{\Nb}{2l}+\gamma$. By the lower bound of
Theorem~\ref{t_est_distr_func_LO},
\[
   \sum_k|\lambda_k|^\gamma = \gamma\int_0^\infty\lambda^{\gamma-1}N(-\lambda;H_\mu)\mathrm{d}\lambda
   \ge \gamma C_1 \int_0^\infty\lambda^{p-1}\,
   \mathcal L\bigl(M_{\mu,\tau_1}(A_1\lambda)\bigr) \mathrm{d}\lambda .
\]
Substituting $t=A_1\lambda$ and using the layer-cake formula,
\begin{equation*}
    \sum_k|\lambda_k|^\gamma \geq   \frac{\gamma C_1}{A_1^p} \int_0^\infty t^{p-1} 
   \mathcal L\bigl(M_{\mu,\tau_1}(t)\bigr)\mathrm{d}t = \frac{\gamma C_1}{pA_1^p}\int_{\R^\Nb}
   \bigl(q^*_{\mu,\tau_1}(x)\bigr)^{p} \mathrm{d}x
\end{equation*}
This gives the lower bound and completes the proof.
\end{proof}

\subsection{Scaling behavior}
Next, we discuss the behavior of the LT-sum when a coupling parameter $\eta$ is introduced. The classical LT estimate is homogeneous 
with respect to $\eta$. For estimates in terms of the Otelbaev function, the situation is more complicated.  We note first that if 
$\mu$ satisfies Condition 
\ref{cond.cap}, then $\eta\mu$ satisfies this condition for all $\eta>0$ as well. Therefore, the operator $H_{\eta\mu}$ and the 
corresponding 
Otelbaev function are well defined. First, we have an automatic corollary of Theorem~\ref{thm: LT_for_LO}.

\begin{corollary}\label{cor_big_coupling}
Let $\mu$ be a Radon measure on $\mathbb{R}^\Nb$ satisfying
Condition~\ref{cond.cap}. For $\eta>0$, let $H_{\eta\mu}$ be the operator
generated by the quadratic form~\eqref{quad_form_LO} with $\mu$ replaced by
$\eta\mu$. Let $\gamma>0$ and assume that
\begin{equation*}
    q^*_{\mu,\eta \tau}
    \in
    L^{\frac{\Nb}{2l}+\gamma}(\mathbb R^\Nb),
    \qquad
    \tau=\theta\varpi(1+\rho).
\end{equation*}
Then the negative spectrum of $H_{\eta\mu}$ is discrete. Moreover, if
$\{\lambda_k(\eta)\}$ denotes the set of  negative eigenvalues of $H_{\eta\mu}$,
counted with multiplicity, then
\begin{equation*}
     \sum_k |\lambda_k(\eta)|^\gamma
     \leq
     C
     \int_{\mathbb{R}^\Nb}
     \left(q_{\mu,\eta \tau}^*(y)\right)^{\frac{\Nb}{2l}+\gamma} \mathrm{d}y,
\end{equation*}
for some constant $C>0$ depending only on $\Nb$, $l$, and $\gamma$.
\end{corollary}

\begin{proof}
    This follows from Theorem \ref{thm: LT_for_LO} and the relation $q_{\eta \mu, \tau}^* = q_{\mu, \eta \tau}^*$.
\end{proof}

\begin{remark}
In contrast to the classical Lieb-Thirring inequality, the upper bound in
Corollary~\ref{cor_big_coupling} does not exhibit a simple homogeneous
dependence on the coupling parameter $\eta$.
In particular, the finiteness of the right-hand side for a given value of $\eta$
does not automatically imply finiteness for larger $\eta>0$; see Example \ref{lem_unpr_big_coupl}.
This feature will be discussed in more detail in the comparison section below.
\end{remark}

\begin{example}\label{lem_unpr_big_coupl}
  For any fixed $\gamma>0$, $\tau>0$, there exists a measure $\mu$ satisfying Condition \ref{cond.cap} such that 
    \begin{equation*}
        q_{\mu,\tau}^* \in L^{\frac{\Nb}{2l} + \gamma}(\mathbb{R}^\Nb),
        \qquad
        q_{\eta\mu,\tau}^* \notin L^{\frac{\Nb}{2l} + \gamma}(\mathbb{R}^\Nb),
    \end{equation*}
    for sufficiently large $\eta>0$.
\end{example}

\begin{proof}
    Fix $\gamma>0$ and $\tau>0$. Let $\{x_k\}_{k\in \mathbb{N}}\subset \mathbb{R}^\Nb$ be a sequence of points such that $|x_k| = 
    2\cdot 4^k$. Set 
    \begin{equation*}
        \mu = \varepsilon \chi_{B}\mathcal{L},
        \qquad
        B = \bigcup_{k\in\mathbb{N}} B_1(x_k)
    \end{equation*}
    where $\mathcal{L}$ is the Lebesgue measure on $\mathbb{R}^\Nb$, $\chi_B$ is the indicator function, and $\varepsilon>0$ is a 
    small number specified later. By \eqref{eq_rest_L_satis_Cond_cap}, $\mu$ satisfies Condition \ref{cond.cap}. By 
    \eqref{est_quant_L}, we know that 
    \begin{equation*}
         \sup_{\Eb\subset B_1(x_k)} Z_{\mathcal{L}}(\Eb) < C_{\Nb,l}
    \end{equation*}
    for any $k\in\mathbb{N}$. We fix $\varepsilon< (C_{\Nb,l}\tau)^{-1}$, and conclude from the last estimate that 
    \begin{equation}\label{est_Phi(L)_less_1/C}
         \sup_{\Eb\subset B_1(x_k)} Z_\mu(\Eb) =\sup_{\Eb\subset B_1(x_k)} Z_{\varepsilon\mathcal{L}}(\Eb) =\varepsilon 
         \sup_{\Eb\subset B_1(x_k)} Z_{\mathcal{L}}(\Eb) < \frac{1}{\tau}.
    \end{equation}
    We set $R_0 = B_4(0)$ and
    \begin{equation*}
        R_k = \{x \in \mathbb{R}^\Nb: \; 4^k \leq |x| < 4^{k+1}\}
    \end{equation*}
    for $k\in\mathbb{N}$. Fix $x\in R_k$. We claim that $Q_{\mu,\tau}(x)$ intersects at least two balls among 
    $\{B_1(x_j)\}_{j\in\mathbb{N}}$. Indeed, if it met at most one ball $B_1(x_{j_0})$, then by capacity monotonicity and 
    \eqref{est_Phi(L)_less_1/C},
    \begin{equation*}
        \sup_{\Eb\subset Q_{\mu,\tau}(x)}Z_\mu(\Eb)\leq \sup_{\Eb\subset B_1(x_{j_0})}Z_\mu(\Eb)<\frac{1}{\tau},
    \end{equation*}
    and this strict bound would persist on a slightly larger cube. This contradicts the definition of $\delta_{\mu,\tau}(x)$ as a 
    supremum; see Definition \ref{def:delta}. The edgelength of $Q_{\mu,\tau}(x)$ is therefore bounded below by the distance between 
    two such balls, 
    that is,
    \begin{equation*}
        \delta_{\mu,\tau}(x)\gtrsim 4^k,
        \qquad x\in R_k,
    \end{equation*}
    and 
    \begin{equation*}
        q^*_{\mu,\tau}(x)=\delta_{\mu,\tau}(x)^{-2l}\lesssim 4^{-2kl},
        \qquad x\in R_k.
    \end{equation*}
    Therefore,
    \begin{equation*}
        \int_{\mathbb{R}^\Nb} \left(q_{\mu,\tau}^*(x)\right)^{\frac{\Nb}{2l} + \gamma} \mathrm{d}x = \sum_{k = 0}^\infty \int_{R_k} 
        \left(q_{\mu,\tau}^*(x)\right)^{\frac{\Nb}{2l} + \gamma} \mathrm{d}x \lesssim \sum_{k=0}^\infty |R_k| \frac{1}{4^{k\Nb}} 
        \cdot 
        \frac{1}{4^{2kl\gamma}}.
    \end{equation*}
    Since $|R_k| \simeq 4^{k\Nb}$ and $\gamma>0$, we conclude that the sum on the right-hand side converges, and therefore 
    $q_{\mu,\tau}^* \in L^{\frac{\Nb}{2l} + \gamma}(\mathbb{R}^\Nb)$.

    Next, we will show that $q_{\eta\mu,\tau}^* \notin L^{\frac{\Nb}{2l} + \gamma}(\mathbb{R}^\Nb)$ for sufficiently large $\eta>0$. 
    Since $\mu$ satisfies Condition \ref{cond.cap}, as we established above,
    \begin{equation*}
        \sup_{\Eb\subset B_1(x_k)} Z_\mu(\Eb) =\sup_{\Eb\subset B_1(x_k)} Z_{\varepsilon\mathcal{L}}(\Eb) > 0
    \end{equation*}
    for all $k\in \mathbb{N}$. Therefore, we can choose $\eta$ large enough so that 
    \begin{equation*}
        \sup_{\Eb\subset B_1(x_k)} Z_{\eta\mu}(\Eb) =\sup_{\Eb\subset B_1(x_k)} Z_{\eta \varepsilon\mathcal{L}}(\Eb) > 
        \frac{1}{\tau}.
    \end{equation*}
    Hence, for every $x\in B_1(x_k)$, the cube $Q_{\eta\mu,\tau}(x)$ cannot contain the ball $B_1(x_k)$. This 
    gives an upper bound for the size of the cube, that is, $\delta_{\eta\mu,\tau}(x) < 4$. Then
    \begin{equation*}
        q_{\eta\mu,\tau}^*(x) > \frac{1}{4^{2l}},
        \qquad
        x \in \bigcup_{k\in \mathbb{N}} B_1(x_k).
    \end{equation*}
    Therefore, it is clear that $q_{\eta\mu,\tau}^* \notin L^{\frac{\Nb}{2l} + \gamma}(\mathbb{R}^\Nb)$.
\end{proof}

\section{Comparison with preceding results}

\subsection{General estimates} In this section, we compare the Lieb-Thirring-type estimates in Theorem \ref{thm: LT_for_LO} with 
previously known results. We will consider operators of 
the form 
\begin{equation*}
    H_{V\nu} = (-\Delta)^l - \mu,
    \qquad
    \mu= V\nu,
\end{equation*}
where $\nu$ is a measure and $V$ is a $\nu$-measurable function. In applications, the measure $\nu$, which may be singular, is 
usually fixed and describes the geometry of the support of the potential, while $V$ plays the role of a density; see, for example, 
\cite{Rozenblum2022,Roz26}. The most interesting  example here is provided by the   measure $\nu$ being  the natural measure on a 
Lipschitz surface of codimension less than $2l$. Generally,  we suppose that $\nu$ and $V$ satisfy the following condition, where 
$\nu$ is required to satisfy an upper Ahlfors 
condition of order $s$.
\begin{condition}\label{cond:nu_V}
    Fix $\gamma>0$. We assume that $\nu$ is a locally finite Borel measure
    satisfying the upper Ahlfors condition of order $s>\Nb-2l$,
    see~\eqref{upper_Ahlfors_condition}. We also assume that $V\geq0$ and
    \begin{equation*}
        V\in L^\varkappa(\mathbb R^\Nb;\nu),
        \qquad
        \varkappa=\frac{s+2l\gamma}{s-\Nb+2l}.
    \end{equation*}
\end{condition}

The following result was proved in \cite{Rozenblum2022}:
\begin{theorem}\label{Thm.LTGR}
    Let $\gamma>0$. Suppose that $\nu$ and $V$ satisfy Condition \ref{cond:nu_V}. Then the negative part of the spectrum of 
    $H_{V\nu}$ 
    is discrete, and its negative eigenvalues $\{\lambda_k\}$ (counted with multiplicity) satisfy
    \begin{equation*}
        \sum_k |\lambda_k|^\gamma \leq C \int_{\mathbb{R}^\Nb} V^\varkappa \mathrm{d}\nu,
    \end{equation*}
    where $C>0$ is independent of $V$.
\end{theorem}

We will set $\mu = V\nu$. First, we will show that Condition \ref{cond.cap} together with the assumption of discreteness is less 
restrictive than Condition \ref{cond:nu_V}. 

\begin{lemma}\label{lem:our_cond_weaker}
    For any  fixed $\gamma>0$,  assume that $\nu$ and $V$ satisfy Condition \ref{cond:nu_V}. Then $\mu: = V\nu$ satisfies Condition 
    \ref{cond.cap}. 
\end{lemma}
\begin{proof}
    Let $x\in \mathbb{R}^\Nb$ and $d>0$. Let $\Eb\subset Q_d(x)$ be a Borel set. By the H\"{o}lder inequality,
    \begin{equation*}
        \mu(\Eb) = \int_{\Eb} V \mathrm{d}\nu \leq \left(\int_{\Eb} V^\varkappa \mathrm{d}\nu\right)^{\frac{1}{\varkappa}} 
        \nu(\Eb)^{1 - 
        \frac{1}{\varkappa}}.
    \end{equation*}
    For 
    \begin{equation*}
        \vartheta=\frac{s}{2l-\Nb+s},
    \end{equation*}
    by Lemma \ref{l: measure_upper_bound_cap},
    \begin{equation*}
        \mu(\Eb) \leq \left(\int_{\Eb} V^\varkappa \mathrm{d}\nu\right)^{\frac{1}{\varkappa}} \nu(\Eb)^{1 - \frac{1}{\vartheta}} 
        \nu(\Eb)^{\frac{1}{\vartheta} - \frac{1}{\varkappa}} \leq C_{\Nb,l} \left(\int_{\Eb} V^\varkappa 
        \mathrm{d}\nu\right)^{\frac{1}{\varkappa}} \,\operatorname{Cap}_l(\Eb) \, \nu(\Eb)^{\frac{1}{\vartheta} - 
        \frac{1}{\varkappa}}.
    \end{equation*}
    Therefore, using $\Eb\subset Q_d(x)$ and Condition \ref{cond:nu_V}, we obtain
    \begin{equation*}
        Z_\mu(\Eb) \leq C_{\Nb,l} \, \|V\|_{L^\varkappa(\mathbb{R}^\Nb;\nu)} \, \nu(Q_d(x)) ^{\frac{1}{\vartheta} - 
        \frac{1}{\varkappa}} \leq C_{\Nb,l } \, \|V\|_{L^\varkappa(\mathbb{R}^\Nb;\nu)} \, d^{s(\frac{1}{\vartheta} - 
        \frac{1}{\varkappa})}.
    \end{equation*}
    Since $\varkappa>\vartheta$, this implies that 
    \begin{equation*}
        \lim_{d \to 0}
        \sup_{x \in \mathbb{R}^\Nb}
        \sup_{\Eb \subset Q_d(x)}
        Z_\mu(\Eb) 
        = 0.
    \end{equation*}
    It follows now that $\mu$ satisfies Condition \ref{cond.cap}.
\end{proof}
The example to follow demonstrates that our Otelbaev function approach, Theorem  \ref{thm: LT_for_LO}, can give a finite estimate for 
the LT-sum, while the integral in the classical estimate is infinite. 

\begin{example}
\label{lem:RT_LT_infinite_ours_finite}

There exist a locally finite Borel measure $\nu$ on $\mathbb R^\Nb$
satisfying the upper Ahlfors condition~\eqref{upper_Ahlfors_condition} with some $s>\Nb-2l$,
and a nonnegative function $V$, such that
\begin{equation*}
    \int_{\mathbb R^\Nb} V^\varkappa \mathrm{d}\nu = \infty,
    \qquad
    \text{for every fixed } \varkappa>0,
\end{equation*}
while for $\mu=V\nu$,
\begin{equation*}
    \int_{\mathbb R^\Nb} \bigl(q^*_{\mu,\tau}(x)\bigr)^{\frac{\Nb}{2l}+\gamma} \mathrm{d}x<\infty
    \qquad\text{for every fixed } \gamma>0 \text{ and } \tau>0.
\end{equation*}
\end{example}

\begin{proof}
We choose a sequence of points $\{x_k\}_{k\in \mathbb{N}}$ such that $|x_k| = 2\cdot 4^k$. Next, we choose a locally finite Borel 
measure $\nu$ and a sequence of Borel sets $\{\Eb_k\}_{k\in \mathbb{N}}$ with the following properties: $\nu$ satisfies the upper 
Ahlfors 
condition~\eqref{upper_Ahlfors_condition} and 
\begin{equation*}
    \Eb_k\subset B_1(x_k),
    \qquad
    \nu(\Eb_k)=\frac{1}{k},
    \qquad
    \lim_{k\to\infty}\operatorname{diam}(\Eb_k) = 0.
\end{equation*}
Define
    \begin{equation*}
        V=\sum_{k\ge 2} \chi_{\Eb_k},
    \end{equation*}
    where $\chi_{\Eb_k}$ is the indicator function. Then, we compute
    \begin{equation*}
        \int_{\mathbb R^\Nb} V^\varkappa \mathrm{d}\nu = \sum_{k\in\mathbb{N}}\nu(\Eb_k)=\sum_{k\in\mathbb{N}}\frac1k=\infty.
    \end{equation*}

Next, we show that $\mu = V\nu$ satisfies Condition \ref{cond.cap}. For an arbitrary point $x$, let $Q_d(x)$ be a cube and let 
$\Eb\subset Q_d(x)$ be a Borel set. For $d>0$ sufficiently small, the cube $Q_d(x)$ intersects at most one of the sets $\Eb_k$; 
otherwise, if there are no such sets, we have $\mu(\Eb)=0$ and there is nothing to prove. Thus, we may assume $\Eb\subset \Eb_k$ for 
some $k$, so that $\mu(\Eb)=\nu(\Eb)$. Since $\nu$ satisfies \eqref{upper_Ahlfors_condition}, Lemma \ref{l: measure_upper_bound_cap} 
gives
    \begin{equation*}
        Z_\mu(\Eb) = Z_\nu(\Eb) \leq C_{\Nb,l} \nu(\Eb)^{\frac{1}{\vartheta}} \leq C_{\Nb,l} \nu(Q_d(x))^{\frac{1}{\vartheta}},
        \qquad
        \vartheta = \frac{s}{s-\Nb+2l}.
    \end{equation*}
    Using \eqref{upper_Ahlfors_condition}, we obtain 
    \begin{equation*}
        Z_\mu(\Eb) \leq C_{\Nb,l} d^{s - \Nb + 2l},
    \end{equation*}
Since $s>\Nb-2l$, this implies
\[
\lim_{d\to0}\sup_{x\in\mathbb R^\Nb}
\sup_{\Eb\subset Q_d(x)}
Z_\mu(\Eb) =0,
\]
and hence $\mu=V\nu$ satisfies Condition~\ref{cond.cap}.
In particular, $q^*_{\mu,\tau}$ is well defined.

Finally, we will show that $q_{\mu,\tau}^* \in L^{\frac{\Nb}{2l} + \gamma}(\mathbb{R}^\Nb)$. We set 
    \begin{equation*}
        R_k = \{x \in \mathbb{R}^\Nb: \; 4^k \leq |x| < 4^{k+1}\}.
    \end{equation*}
Let us estimate $\delta_{\mu,\tau}$ from below. Since $\mu$ satisfies Condition \ref{cond.cap} and 
    $$
        \lim_{j\rightarrow\infty} \operatorname{diam}(\Eb_j) = 0,
    $$
there exists $k_0\in\mathbb{N}$ such that 
\begin{equation*}
    \sup_{\Eb\subset \Eb_j}Z_\mu(\Eb)<\frac{1}{\tau}
\end{equation*}
for every $j\geq k_0$. Fix $x\in R_k$ with $k\geq k_0$. We claim that $Q_{\mu,\tau}(x)$ has a non-void intersection with some of the 
sets  $\Eb_j$ with $j\neq 
k$. Indeed, otherwise it meets at most $\Eb_k$, and since $k\geq k_0$, capacity monotonicity gives
\begin{equation*}
    \sup_{\Eb\subset Q_{\mu,\tau}(x)}Z_\mu(\Eb)\leq \sup_{\Eb\subset \Eb_k}Z_\mu(\Eb)<\frac{1}{\tau},
\end{equation*}
and this strict bound would persist on a slightly larger cube. This contradicts the definition of $\delta_{\mu,\tau}(x)$. Hence
    \begin{equation*}
        \delta_{\mu,\tau}(x)\gtrsim 4^k,
        \qquad x\in R_k,\ k\geq k_0
    \end{equation*}
    and 
    \begin{equation*}
        q^*_{\mu,\tau}(x)=\delta_{\mu,\tau}(x)^{-2l}\lesssim 4^{-2kl},
        \qquad x\in R_k,\ k\ge k_0.
    \end{equation*}
    Therefore,
    \begin{equation*}
        \int_{\mathbb{R}^\Nb} \left(q_{\mu,\tau}^*(x)\right)^{\frac{\Nb}{2l} + \gamma} \mathrm{d}x \lesssim \int_{|x| < 4^{k_0+1}} 
        \left(q_{\mu,\tau}^*(x)\right)^{\frac{\Nb}{2l} + \gamma} \mathrm{d}x + \sum_{k>k_0} |R_k| \frac{1}{4^{k\Nb}} \cdot 
        \frac{1}{4^{2kl\gamma}}.
    \end{equation*}
    The first term on the right-hand side is finite. Since $|R_k| \simeq 4^{k\Nb}$ and $\gamma>0$, the second term is finite as well. 
    This completes the proof.
\end{proof}
\subsection{Surface LT estimates} In \cite{FrankLaptev} and, further, \cite{Rozenblum2022}, LT-type estimates were established for 
measure-potentials of the form $\mu=V\nu$, where $\nu$ is a singular measure, supported on a hyperplane, in \cite{FrankLaptev}, or on 
set of dimension less than $\Nb$, in \cite{Rozenblum2022}. In the present subsection, we discuss how these estimates (we call them 
'surface LT bounds', since the case of measures supported on a surface is the most important one) are related with general Otelbaev 
function ones in Theorem \ref{thm: LT_for_LO}.

Our aim is to  show that, for Ahlfors regular measures, the finiteness of the
integral in Theorem \ref{Thm.LTGR}  implies the finiteness in the
bound \eqref{Otelbaev.LT} in terms of the Otelbaev function (in particular, when
$\nu=\mathcal L$ and $\Sigma=\supp\nu=\mathbb R^\Nb$, this recovers the usual
absolutely continuous case $\mu=Vdx$. As always, we do not address the question of
sharp constants.

We now use the geometric estimate from Lemma~\ref{lem:tube_estimate} to compare the full-space integral of the Otelbaev function with 
the corresponding surface integral over $\Sigma$. This gives the desired control of the Otelbaev bound by the natural 
$L^\varkappa(\Sigma;\nu)$ norm of the density $V$.

\begin{proposition}\label{prop:qstar_control_by_surface_Lp}
Let \(\nu\) be a locally finite Borel measure on \(\mathbb R^\Nb\), let
\(\Sigma=\operatorname{supp}\nu\), and assume that \(\nu\) is
Ahlfors regular for some \(s>\Nb-2l\), that is, \eqref{upper_Ahlfors_condition} and \eqref{lower_Ahlfors_condition} hold. Let 
\(V\geq0\) and \(V\in L^\varkappa(\Sigma;\nu)\), where
\begin{equation*}
    \varkappa=\frac{s+2l\gamma}{s-\Nb+2l},
    \qquad \gamma>0,
\end{equation*}
and let \(\mu=V\nu\). Then, for every \(\tau>0\),
\begin{equation}\label{eq:qstar_surface_Lp_bound}
    \int_{\mathbb R^\Nb}
    \bigl(q^*_{\mu,\tau}(x)\bigr)^{\frac{\Nb}{2l}+\gamma} \mathrm{d}x
    \leq
    C
    \int_{\Sigma} V^\varkappa \mathrm{d}\nu,
\end{equation}
where \(C\) depends only on
\(\Nb,l,s,\gamma,\tau,c_\nu,\mathcal A_\nu\).
\end{proposition}

\begin{proof}
By Lemma~\ref{lem:our_cond_weaker} and the upper Ahlfors bound \eqref{upper_Ahlfors_condition}, the measure \(\mu=V\nu\) satisfies
Condition~\ref{cond.cap}, so the Otelbaev function \(q^*_{\mu,\tau}\) is
well defined. By Remark~\ref{rem:infinite_d}, we have either \(\delta_{\mu,\tau}\equiv\infty\) (in which
case the assertion is trivial) or \(\delta_{\mu,\tau}(x)<\infty\) for all
\(x\in\mathbb R^\Nb\); we assume the latter. Set
\begin{equation*}
    p=\frac{\Nb}{2l}+\gamma,
    \qquad
    a=2lp=\Nb+2l\gamma,
\end{equation*}
so that \((q_{\mu,\tau}^*(x))^p=\delta_{\mu,\tau}(x)^{-a}\).

Since \(\mu\) is supported on \(\Sigma\), if
\(Q_d(x)\cap\Sigma=\emptyset\), then 
\begin{equation*}
    \sup_{\Eb\subset Q_d(x)}Z_\mu(\Eb)=0,
\end{equation*}
and consequently $\delta_{\mu,\tau}(x)\geq d$. Therefore,
\begin{equation}\label{eq:delta_dist_sigma}
    \delta_{\mu,\tau}(x)\geq 2\operatorname{dist}_\infty(x,\Sigma),
    \qquad x\in\mathbb R^\Nb.
\end{equation}

For $r>0$ define
\begin{equation*}
    \Omega_r=\{x\in\mathbb R^\Nb:\delta_{\mu,\tau}(x)<r\},
    \qquad
    \mathcal E_r=\{y\in\Sigma:\delta_{\mu,\tau}(y)<r\}.
\end{equation*}
Let $x\in\Omega_r$. By \eqref{eq:delta_dist_sigma},
\begin{equation*}
    2\operatorname{dist}_\infty(x,\Sigma)\leq\delta_{\mu,\tau}(x)<r.
\end{equation*}
Since $\Sigma$ is closed, there exists $y\in\Sigma$ with
\begin{equation*}
    \|x-y\|_\infty=\operatorname{dist}_\infty(x,\Sigma)<\frac{r}{2}.
\end{equation*}
By Lemma~\ref{lem:continuity_d},
\begin{equation*}
    \delta_{\mu,\tau}(y)\leq\delta_{\mu,\tau}(x)+2\|x-y\|_\infty<r+2\cdot \frac{r}{2}=2r,
\end{equation*}
hence $y\in\mathcal E_{2r}$. In the notation of
\eqref{def:B_sets}, this shows that
\begin{equation}\label{eq:Omega_tube}
    \Omega_r\subset
    \Bc_{\mathbb R^\Nb}\!\bigl(\mathcal E_{2r},\,r/2\bigr).
\end{equation}

Lemma~\ref{lem:tube_estimate} applied to $\mathcal{E}_{2r}$ and
$r/2$ gives
\begin{equation}\label{eq:tube_applied}
    \bigl|\Bc_{\mathbb R^\Nb}(\mathcal E_{2r},\,r/2)\bigr|
    \leq
    \frac{4^\Nb}{c_\nu}\,\Bigl(\frac{r}{2}\Bigr)^{\Nb-s}\,
    \nu\bigl(\Bc_\Sigma(\mathcal E_{2r},\,r/2)\bigr)
    =
    \frac{2^{\Nb+s}}{c_\nu}\,r^{\Nb-s}\,
    \nu\bigl(\Bc_\Sigma(\mathcal E_{2r},\,r/2)\bigr).
\end{equation}
If $y\in \Bc_\Sigma(\mathcal E_{2r},r/2)$, there exists $z\in\mathcal E_{2r}$ such that 
\begin{equation*}
    \|y-z\|_\infty< \frac{r}{2}.
\end{equation*}
By Lemma~\ref{lem:continuity_d},
\begin{equation*}
    \delta_{\mu,\tau}(y)\leq\delta_{\mu,\tau}(z)+2\|y-z\|_\infty<2r+2\cdot \frac{r}{2}=3r,
\end{equation*}
that is, $y\in\mathcal E_{3r}$. Therefore, we conclude that $\Bc_\Sigma(\mathcal E_{2r},r/2) \subset \mathcal{E}_{3r}$. Hence, by 
\eqref{eq:Omega_tube} and \eqref{eq:tube_applied},
\begin{equation}\label{eq:level_set_estimate}
    |\Omega_r|\leq
    \frac{2^{\Nb+s}}{c_\nu}\,r^{\Nb-s}\,\nu(\mathcal E_{3r}).
\end{equation}
By the layer-cake formula, we know
\begin{equation*}
    \int_{\mathbb R^\Nb}\delta_{\mu,\tau}(x)^{-a}\mathrm{d}x
    = a\int_0^\infty t^{a-1}\left|\left\{x \in \mathbb{R}^{\Nb}: \; \frac{1}{\delta_{\mu,\tau}(x)}>t\right\}\right|\mathrm{d}t.
\end{equation*}
Therefore, changing variables $t=1/r$ and \eqref{eq:level_set_estimate} give
\begin{multline*}
    \int_{\mathbb R^\Nb}\delta_{\mu,\tau}(x)^{-a}\mathrm{d}x
    = a\int_0^\infty r^{-a-1}|\Omega_r|\mathrm{d}r
    \leq
    \frac{a2^{\Nb+s}}{c_\nu}
    \int_0^\infty r^{-a-1+\Nb-s} \nu(\mathcal E_{3r})\mathrm{d}r  \\
    =
    \frac{a2^{\Nb+s}}{c_\nu}
    \int_0^\infty r^{-s-2l\gamma-1} \nu(\mathcal E_{3r})\mathrm{d}r,
\end{multline*}
Further the change of variable $r'=3r$ yields
\begin{equation*}
    \int_{\mathbb R^\Nb}\delta_{\mu,\tau}(x)^{-a}  \mathrm{d}x
    \leq \frac{a2^{\Nb+s}3^{s+2l\gamma}}{c_\nu} \int_0^\infty (r')^{-s-2l\gamma-1}\nu(\mathcal E_{r'}) \mathrm{d}r'.
\end{equation*}
Applying the layer-cake formula on $(\Sigma, \nu)$ and arguing as above, we have
\begin{equation*}
    \int_\Sigma\delta_{\mu,\tau}(y)^{-(s+2l\gamma)}\mathrm{d}\nu(y) = (s+2l\gamma)\int_0^\infty (r')^{-s-2l\gamma-1}\nu(\mathcal 
    E_{r'})\mathrm{d}r',
\end{equation*}
we obtain
\begin{equation}\label{eq:full_to_surface_delta}
    \int_{\mathbb R^\Nb}\left(q^*_{\mu,\tau}(x)\right)^p \mathrm{d}x 
    =
    \int_{\mathbb R^\Nb}\delta_{\mu,\tau}(x)^{-a} \mathrm{d}x 
    \leq
    C_1\int_\Sigma\delta_{\mu,\tau}(y)^{-(s+2l\gamma)} \mathrm{d}\nu(y),
\end{equation}
where
\begin{equation}
    C_1=\frac{(\Nb+2l\gamma)\,2^{\Nb+s}\,3^{s+2l\gamma}}{c_\nu(s+2l\gamma)}.
\end{equation}

Set
\begin{equation*}
    \vartheta=\frac{s}{s-\Nb+2l}.
\end{equation*}
Let $d>0$, $x\in \Sigma$, and $\Eb\subset Q_d(x)$ be a Borel set. By Lemma \ref{l: measure_upper_bound_cap},
    \begin{equation*}
        \nu(\Eb)^{1 - \frac{1}{\vartheta}} \leq C_{\Nb,l,\mathcal{A}_\nu}  \operatorname{Cap}_l(\Eb).
    \end{equation*}
    Then, the H\"{o}lder inequality gives
    \begin{multline*}
        Z_\mu(\Eb)
        =
        \frac{\int_{\Eb} V(x) \mathrm{d}\nu(x)}{\operatorname{Cap}_l(\Eb)} 
        \leq
        \frac{\left(\int_{\Eb} V^{\vartheta}(x) \mathrm{d}\nu(x)\right)^{1/\vartheta} 
        \nu(\Eb)^{1-\frac{1}{\vartheta}}}{\operatorname{Cap}_l(\Eb)} \\
        \leq
        C_{\Nb,l,\mathcal{A}_\nu} \|V\|_{L^{\vartheta}(\Eb;\nu)}
        \leq
        C_{\Nb,l,\mathcal{A}_\nu} \|V\|_{L^{\vartheta}(Q_d(x);\nu)}.
    \end{multline*}
    Therefore, 
    \begin{multline}\label{eq:tilded_smaller_than_delta}
        \widetilde{d}(x)
        :=
        \sup\left\{d>0:
        C_{\Nb,l,\mathcal{A}_\nu} \|V\|_{L^{\vartheta}(Q_d(x);\nu)}
        <
        \frac{1}{\tau}
        \right\}\\
        \leq 
         \sup\left\{
        d>0:\;
        \sup_{\substack{\Eb\subset Q_d(x)}}
        Z_\mu(\Eb)
        <
        \frac{1}{\tau}
        \right\} = \delta_{\mu,\tau}(x).
    \end{multline}
Since $\widetilde d(x) \leq \delta_{\mu,\tau}(x) <\infty$, the definition of $\widetilde d(x)$ implies
that, for every $d>\widetilde d(x)$, 
\begin{equation*}
    C_{\Nb,l,\mathcal A_\nu}
    \left(\int_{Q_{d}(x)} V^\vartheta \mathrm{d}\nu\right)^{1/\vartheta}
    \geq
    \frac1\tau.
\end{equation*}
Hence
\begin{equation*}
    1
    \leq
    \left(\tau C_{\Nb,l,\mathcal A_\nu}\right)^\vartheta
    \int_{Q_{d}(x)} V^\vartheta \mathrm{d}\nu .
\end{equation*}
Dividing by $\nu(Q_{d}(x))$, we get
\begin{equation*}
    \nu(Q_{d}(x))^{-1}
    \leq
    \left(\tau C_{\Nb,l,\mathcal A_\nu}\right)^\vartheta
     \frac{1}{\nu(Q_{d}(x))}  \int_{Q_{d}(x)} V^\vartheta \mathrm{d}\nu.
\end{equation*}
    Introduce the maximal operator
    \begin{equation*}
        \Mt_\nu[f](x) = \sup_{d>0}\frac{1}{\nu(Q_d(x))}\int_{Q_d(x)}|f| \mathrm{d}\nu.
    \end{equation*}
    Therefore, we conclude that, for every $d>\widetilde d(x)$, 
    \begin{equation*}
        \nu(Q_{d}(x))^{-1}
        \leq
        \left(\tau C_{\Nb,l,\mathcal A_\nu}\right)^\vartheta
        \Mt_\nu[V^\vartheta](x).
\end{equation*}
By the upper Ahlfors bound,
\begin{equation*}
    d^{-s}
    \leq
    \frac{\mathcal A_\nu}{\nu(Q_{d}(x))}
    \leq
    \mathcal A_\nu
    \left(\tau C_{\Nb,l,\mathcal A_\nu}\right)^\vartheta
        \Mt_\nu[V^\vartheta](x),
        \qquad
        \text{for } d>\widetilde d(x).
\end{equation*}
Letting $d\downarrow \widetilde d(x)$, we obtain
\begin{equation*}
    \widetilde d(x)^{-s}
    \leq
    C_{\Nb,l,\mathcal A_\nu,\tau,s}
    \Mt_\nu[V^\vartheta](x).
\end{equation*}
Due to \eqref{eq:tilded_smaller_than_delta}, 
\begin{equation*}
    \delta_{\mu,\tau}(x)^{-s}
    \leq
    C_{\Nb,l,\mathcal A_\nu,\tau,s}
    \Mt_\nu[V^\vartheta](x).
\end{equation*}
Inserting this into \eqref{eq:full_to_surface_delta}, we obtain
\begin{equation*}
    \int_{\mathbb R^\Nb}\left(q^*_{\mu,\tau}(x)\right)^p \mathrm{d}x  
    \leq
    C_{\Nb,l,c_\nu,\mathcal{A}_\nu,\tau,s,\gamma} \int_\Sigma \left(\Mt_\nu[V^\vartheta](x)\right)^{\frac{s + 2l\gamma}{s}} 
    \mathrm{d}\nu(x).
\end{equation*}
Since $\gamma>0$, $(s + 2l\gamma)/s>1$. Due to the Ahlfors condition, $\nu$ is doubling. Therefore, by the
Hardy-Littlewood maximal theorem, 
\begin{equation*}
    \int_{\mathbb R^\Nb}\left(q^*_{\mu,\tau}(x)\right)^p \mathrm{d}x  
    \leq
    C_{\Nb,l,c_\nu,\mathcal{A}_\nu,\tau,s,\gamma} \int_\Sigma \left(V(x)\right)^{\vartheta \cdot \frac{s + 2l\gamma}{s}} 
    \mathrm{d}\nu(x).
\end{equation*}
Recalling the definitions of $\varkappa$, $\vartheta$, and $p$, we obtain \eqref{eq:qstar_surface_Lp_bound}.
\end{proof}

\begin{remark}
Proposition~\ref{prop:qstar_control_by_surface_Lp} shows that, as long as sharp
constants are not considered, Theorem~\ref{thm: LT_for_LO} recovers the
estimates of \cite{FrankLaptev}, \cite{Rozenblum2022}, and the classical
absolutely continuous case \cite{LiebThirring1976}. The converse implication
is false in general, as shown by Example~\ref{lem:RT_LT_infinite_ours_finite}.

Due to Lemma~\ref{lem:our_cond_weaker}, the class of admissible potentials in
our approach is larger than the class considered in \cite{Rozenblum2022}.
Moreover, for $s$-Ahlfors regular measures, Proposition~\ref{prop:qstar_control_by_surface_Lp}
shows that our upper bound is sharper in the following sense: finiteness of
the upper bound in \cite{Rozenblum2022} implies finiteness of our upper bound,
whereas the converse is false in general by Example~\ref{lem:RT_LT_infinite_ours_finite}. For measures satisfying only an upper 
Ahlfors condition, the two upper bounds
are not directly comparable, and in this generality the results should be
regarded as complementary.
\end{remark}

As an immediate consequence of the previous proposition, we reproduce the classical large-coupling version.

\begin{corollary}\label{lem:LT_classical_implies_ours_large_coupling}
Let $\tau>0$, let $\gamma>0$, and let $V\in L^{\frac{\Nb}{2l}+\gamma}(\mathbb{R}^\Nb)$ be a non-negative function. Define the measure 
$\mu = V\mathcal{L}$. Then, for every $\eta>0$,
\begin{equation*}
    \|q^*_{\eta\mu,\tau}\|_{L^{\frac{\Nb}{2l}+\gamma}(\mathbb R^\Nb)}
\le A_{\Nb,l,\tau,\gamma}\,\|\eta V\|_{L^{\frac{\Nb}{2l}+\gamma}(\mathbb R^\Nb)}
\end{equation*}
In particular, whenever the classical Lieb-Thirring bound is finite, the corresponding bound in terms of the Otelbaev function is 
finite as well.
\end{corollary}

\begin{proof}
Apply the previous lemma with the measure $\eta\mu=\eta V\mathcal{L}$.
\end{proof}

\begin{remark}
In the classical Lieb-Thirring inequality for a non-negative potential $V$,
the upper bound is expressed solely in terms of the norm
$\|V\|_{L^{\frac{\Nb}{2l}+\gamma}(\mathbb R^\Nb)}$.
Consequently, if the upper bound is finite for $\eta V$ for some $\eta>0$,
then it remains finite for all $\eta>0$.

The situation is different for the estimate in terms of the Otelbaev function $q^*_{\mu,\tau}$, which takes into account the 
spreading of the measure. As shown in Example~\ref{lem_unpr_big_coupl}, the finiteness of
$\|q^*_{\eta\mu,\tau}\|_{L^{\frac{\Nb}{2l}+\gamma}(\mathbb R^\Nb)}$
may depend essentially on the coupling parameter $\eta$.
In particular, there exist measures $\mu$ such that
\begin{equation*}
    q^*_{\mu,\tau}\in L^{\frac{\Nb}{2l}+\gamma}(\mathbb R^\Nb),
\qquad
q^*_{\eta\mu,\tau}\notin L^{\frac{\Nb}{2l}+\gamma}(\mathbb R^\Nb),
\end{equation*}
for sufficiently large $\eta$.

At the same time, $q^*_{\eta\mu,\tau} \in L^{\frac{\Nb}{2l}+\gamma}(\mathbb R^\Nb)$ implies $q^*_{\mu,\tau}\in 
L^{\frac{\Nb}{2l}+\gamma}(\mathbb R^\Nb)$. Thus, the dependence of $q^*_{\mu,\tau}$ on the coupling constant provides a
finer description of the spectral behaviour in the regime of large coupling.
\end{remark}


\end{document}